\documentclass[pdflatex,sn-mathphys-ay]{sn-jnl}

\usepackage{graphicx}%
\usepackage{multirow}%
\usepackage{amsmath,amssymb,amsfonts}%
\usepackage{amsthm}%
\usepackage{mathrsfs}%
\usepackage[title]{appendix}%
\usepackage{xcolor}%
\usepackage{textcomp}%
\usepackage{manyfoot}%
\usepackage{booktabs}%
\usepackage{algorithm}%
\usepackage{algorithmicx}%
\usepackage{algpseudocode}%
\usepackage{listings}%

\newcommand\be{\begin{equation}}
\newcommand\ee{\end{equation}}
\newcommand\bea{\begin{eqnarray}}
\newcommand\eea{\end{eqnarray}}
\newcommand\beaa{\begin{eqnarray*}}
\newcommand\eeaa{\end{eqnarray*}}
\newcommand\beba{\begin{equation}\left\{\begin{array}{rcl}}
\newcommand\eeba{\end{array}\right.\end{equation}}
\newcommand\bebaa{\begin{equation*}\left\{\begin{array}{rcl}}
\newcommand\eebaa{\end{array}\right.\end{equation*}}

\newcommand\blue{\color{black}}
\newcommand\red{\color{black}}

\newcommand\R{\mathbb{R}}
\newcommand\bR{{\mathbb{R}}}

\newcommand\bN{{\mathbb{N}}}
\newcommand\bS{{\mathbb{S}}}

\newcommand\cH{{\mathcal{H}}}

\newcommand\cP{{\mathcal{P}}}

\newcommand{\eps}{\epsilon}

\newcommand{\dsum}{\displaystyle \sum}

\newcommand{\dlim}{\displaystyle \lim}

\theoremstyle{thmstyleone}
\newtheorem{lemma}{Lemma}
\newtheorem{assumption}{Assumption}
\newtheorem{corollary}{Corollary}

\theoremstyle{thmstyleone}%
\newtheorem{theorem}{Theorem}
\newtheorem{proposition}[theorem]{Proposition}%

\theoremstyle{thmstyletwo}%
\newtheorem{example}{Example}%
\newtheorem{remark}{Remark}%

\theoremstyle{thmstylethree}%

\begin{document}
\title[Reaction, diffusion and nonlocal interactions in high-dimensional space]{Reaction, diffusion and nonlocal interactions in high-dimensional space}



\author[1]{\fnm{Hiroshi} \sur{Ishii}}\email{hiroshi.ishii@es.hokudai.ac.jp}

\author*[2]{\fnm{Yoshitaro} \sur{Tanaka}}\email{yoshitaro.tanaka@gmail.com}


\affil[1]{\orgdiv{Research Center of Mathematics for Social Creativity, Research Institute for Electronic Science}, \orgname{Hokkaido University}, \orgaddress{\street{Kita 12, Nishi 7, Kita-ku}, \city{Sapporo}, \postcode{060-0812}, \state{Hokkaido}, \country{Japan}}}

\affil*[2]{\orgdiv{Department of Complex and Intelligent Systems, School of Systems Information Science}, \orgname{Future University Hakodate}, \orgaddress{\street{116-2 Kamedanakano-cho}, \city{Hakodate}, \postcode{ 041-8655}, \state{Hokkaido}, \country{Japan}}}



\abstract{
{\red In this paper we consider the mathematical relationship between nonlocal interactions of convolution type and multiple diffusive substances in high dimensions.
Motivated by that the nonlocal evolution equations reproduce similar patterns to those in reaction-diffusion systems,}
{\red we approximate nonlocal interactions in evolution equations by the solution to a reaction-diffusion system in any dimensional Euclidean space.}
The key aspect of this approach is that {\red any absolutely integrable radial}  kernels can be approximated by a linear combination of specific Green functions. 
This enables us to demonstrate that any nonlocal interactions of convolution type can be approximated by a linear sum of auxiliary diffusive substances.
{\red Moreover, we show that the parameters in the reaction-diffusion system can be specified depending on the kernel shape up to three dimensions.}
Our results establish a connection between a broad class of nonlocal interactions and diffusive chemical reactions in dynamical systems.
}




\maketitle

\backmatter

\section{Introduction}\label{sec:intro}
Various interactions plays a crucial role in phenomena such as developmental processes of multicellular organisms, behavior of biological populations, cell migration and neuronal information processing.
The evolutionary dynamics of each factor involved in these phenomena depend on their interactions.
These interactions give rise to spatiotemporal patterns and complex behaviors in each system.
In particular, some interactions affect the distant objects globally in space as observed in the aforementioned phenomena.
Such interactions are often referred as to nonlocal interactions or long-range interactions.
The presence of the nonlocal interactions has been suggested by biological experiments in various contexts, including neural firing in the brain, pigment cells in skin of zebrafish, cell migration and cell adhesion.

\cite{K1953} experimentally demonstrated that the light response of a ganglion cell in the receptive field of the cat brain exhibits local excitation and lateral inhibition.
This experimental result can be represented as a function, as shown in Fig. \ref{fig1} (a), where local excitation and lateral inhibition correspond to positive and negative values, respectively. 
This function is referred to as the Mexican-hat function due to its shape, or the LALI function (standing for the local activation and lateral inhibition ) in context of pattern formation.
\cite{NTKK2009} experimentally demonstrated that the presence of local inhibition and lateral activation among the pigment cells in the skin of the zebrafish through laser ablation control experiments.
\cite{HWLNHPK2014} further showed that pigment cells in skin of zebrafish extend cellular projections 
to transmit survival signals across the pigment stripe patterns.
These studies on  pigment cell interactions and skin patterns formation in zebrafish are comprehensively reviewed by \cite{WK2015}. 
\cite{MH2015} experimentally demonstrated that variations in the strength of cell adhesion molecules in artificially cultured cells lead to changes in adhesion surface patterns.
It is also suggested that these cultured cells may sense the surrounding cell density by extending cellular projections longer than their body size.
\begin{figure}[bt]
    \begin{center}
    \begin{tabular}{ccc}
\includegraphics[width=5cm, bb=0 0 359 372]{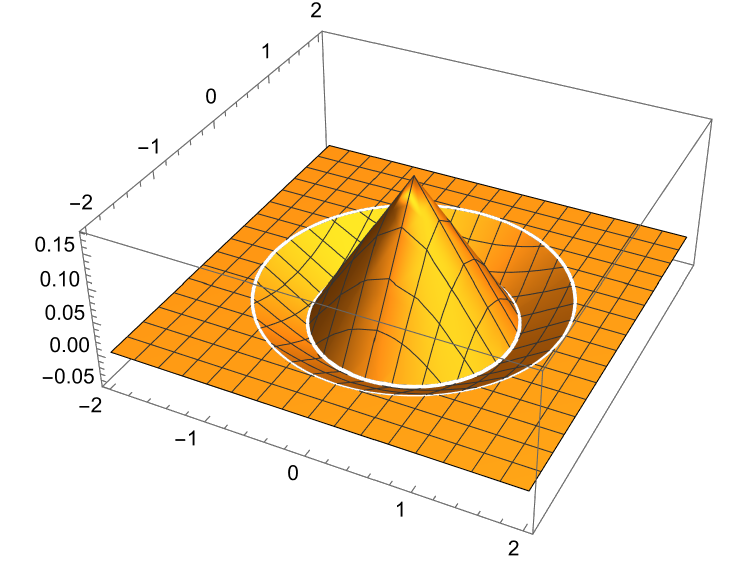}
&\includegraphics[width=5cm, bb=0 0 804 643]{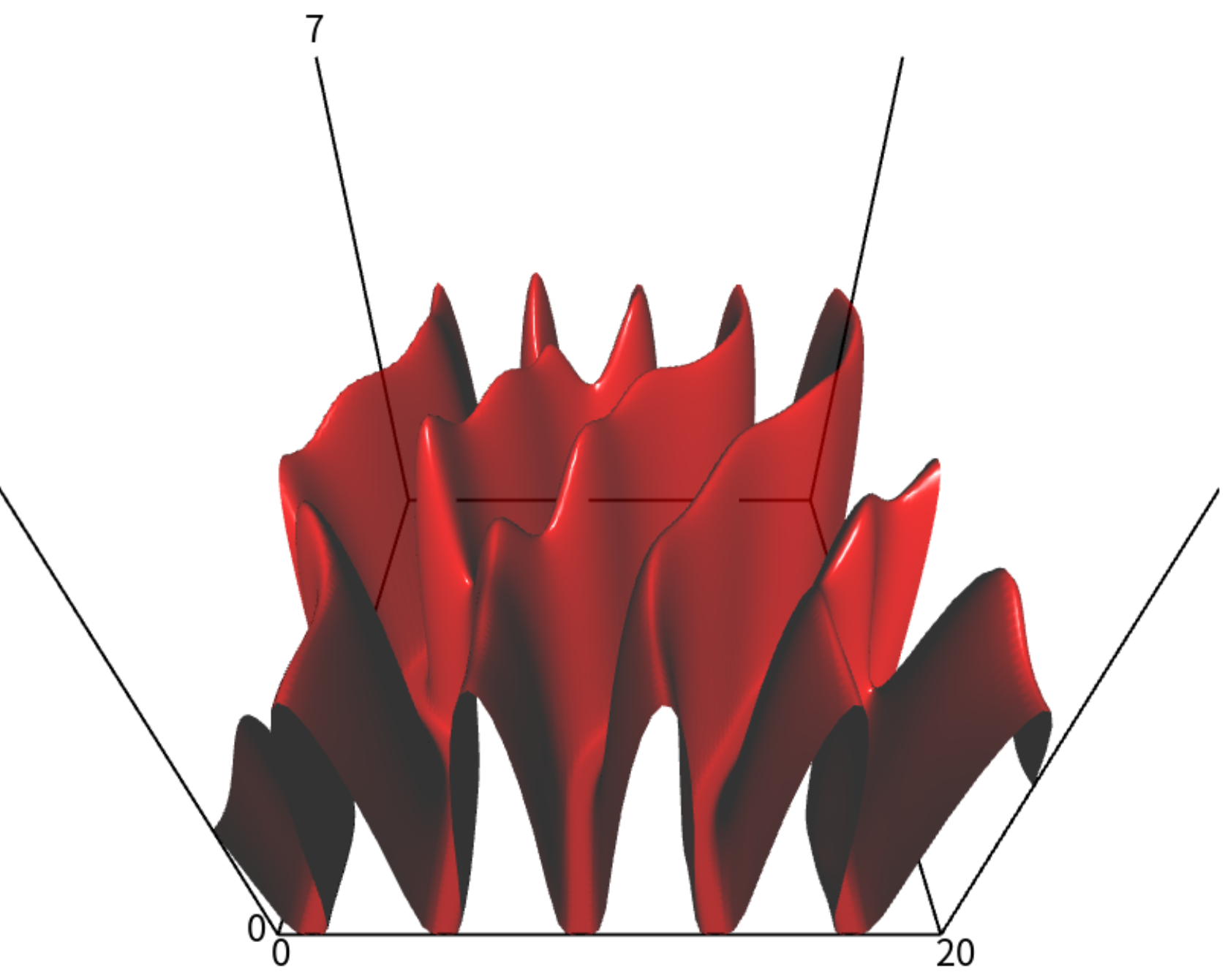}\\
(a)&(b)
    \end{tabular}
    \end{center}
    \caption{Profile of a compact Mexican-hat function and pattern of a solution to \eqref{eq:non-growth}.
    $K(x) = \mu ( a_1/(\pi R_1^3) (R_1 - |x|)\chi_{B(R_1)}(x) - a_2/(\pi R_2^3) (R_2 - |x|)\chi_{B(R_2)}(x) )$ where $0<R_1<R_2$ is radius, $B(R)$ is a ball with radius $R$, $\chi_{B(R)}$ is the characteristic function, and $a_1, a_2$ and $\mu$ are constant. 
    (a) $R_1=1, R_2=1.5,a_1=a_2=\mu=1$.
    (b) A solution of \eqref{eq:non-growth} with $R_1=1.5, R_2=3, a_1=1.2, a_2=1, \mu=50$, and $f(u)=2u(1-u^2)$.
    }
    \label{fig1}
\end{figure}

The nonlocal interaction is often modeled by the spatial convolution with an appropriate integral kernel representing distance-dependent weight and a variable representing the factor or density of individual organism. 
Let $u=u(t,x)$ be the unknown variable at position $x$ at time $t\ge 0$, and we assume that $K \in L^1(\R^n)$ is a radial integral kernel.
In this paper, we mainly treat the nonlocal interactions described by the following form:
\begin{equation}\label{ex:kernel}
		(K*u)(t,x) := \int_{\bR^n} K(x-y)u(t,y)dy.
\end{equation}
Various mathematical models have been proposed and analyzed from the aforementioned biological backgrounds.

\cite{Amari} proposed a neural field equation of nonlocal type with the convolution of the Mexican-hat kernel and the Heaviside function, and derived the condition for the existence of traveling wave solutions. 
\cite{CSB} derived the equation of motion for the interface dynamics of the planar neural field, and demonstrated the numerical simulations and linear stability analysis.
The following model has been proposed to describe the behavior including nonlocal dispersion when an individual organism makes a long-range jump:
\begin{equation}\label{ex:nonlocal}
	u_t = d\Delta u + K*u + f(u),
\end{equation}
where $d \ge 0$ is the diffusion coefficient and $f$ is the reaction term.
The nonlocal model for the plant dispersal has been proposed by \cite{AIM2018}.
The existence and wave speed of the traveling wave solutions to this type of model have been analyzed by \cite{BFRW1997,CD,EGIW,HMMV}.
In particular, \cite{EI} have proposed equations of motion for the interactions between pulse solutions or traveling wave solution to this type of models with the sign changed kernels.
Additionally, a continuation method have been proposed that can convert the spatially discretized models into the form of \eqref{ex:nonlocal} while conserving the discreteness information by \cite{EISTWY}.
From this method it has been theoretically shown that the intercellular interactions such as diffusion, lateral inhibition on uniform and nonuniform lattice, and signal transduction through cellular projections can be described in the form of \eqref{ex:kernel}.

Mathematical models of nonlocal saturation and nonlocal growth rate have been proposed by \cite{Berestycki} and \cite{NTY1}, respectively:
\begin{align}
	&u_t = d\Delta u + (1- K*u )u,\label{eq:non-satu}\\
	&u_t = d\Delta u +  (K*u )u + f(u).\label{eq:non-growth}
\end{align}
The stationary solution and traveling wave solution of \eqref{eq:non-satu} have been analyzed in \cite{Berestycki}.
It has been rigorously shown that the instability for the stationary constant solution induced by a nonlocal interaction can be regarded as the diffusion-driven instability proposed by \cite{T1952} in the nonlocal evolution equations including \eqref{eq:non-growth} through the reaction-diffusion approximation by \cite{NTY1}.
Figure \ref{fig1} (b) shows a numerical example of the pattern.
\cite{Kondo} proposed a nonlocal evolution equation that combines the nonlocal interaction and the cut function motivated by pattern formations.
This paper reported that the model can replicate various patterns such as those observed in animal skins
simply by changing the shape of the integral kernel even though the equation has only one variable.

\cite{MH2015} and \cite{CMSTT2019} proposed the following mathematical model for the cell adhesion and cell sorting phenomena: 
\begin{align}\label{eq:cell-ad}
	u_t = \Delta u^m - \nabla \cdot (u (1-u) \nabla( K*u ) ),
\end{align}
where $m \ge 1$ is a constant.
Integral kernels with compact support as in Fig. \ref{fig1} (a) have been proposed to formulate the cell-cell interactions by \cite{CMSTT2019}.
Moreover, it has been reported that the above model can replicate the cell adhesion phenomena qualitatively and almost quantitatively.

As discussed above, nonlocal interactions can directly model interactions between cells, the potentials between particles, and the intrinsic interactions in density or concentration fields that derive pattern formation by appropriately selecting the shape of the integral kernel.
These nonlocal evolution equations can reproduce the various patterns depending on the shape of the integral kernel.
For examples, by using the Mexican-hat kernel as shown in Fig. \ref{fig1} (a), a scalar nonlocal evolution equation can generate the pattern when diffusion-induced instability  proposed by \cite{T1952} occurs. 
The relationships between differential equations
and nonlocal interactions have been investigated.
The relationship between nonlocal interactions and differential operators has been shown through the Taylor expansion by \cite{Murray}.
\cite{EIKMT} proposed a methodology to extract the pattern-generating information from an arbitrary dimensional network with spatial interactions by representing it as the shape of an integral kernel in a convolution.
This methodology succeeded in deriving the Mexican-hat kernel from general reaction-diffusion system.
\cite{NTY1,NTY2} demonstrated that in one-dimensional spaces, the nonlocal evolution equations with arbitrary even kernel can be approximated by a reaction-diffusion system with multiple auxiliary factors.
This theory reveals that various types of nonlocal interactions can be effectively described using multi-component systems introducing diffusive auxiliary factors.
However,
nonlocal interactions are frequently observed in high-dimensional spatial settings in biological phenomena.

Motivated by these backgrounds, as a natural extension to the study with one-dimensional case, we aim to investigate the relationship between nonlocal interactions of convolution type and reaction-diffusion system with multiple components in high-dimensional spaces.
We approximate nonlocal evolution equations with various integral kernels by a reaction-diffusion system introduced auxiliary diffusive substances in a singular limit.
We show that the solutions to nonlocal evolution equations with any radial kernel can be approximated by that to a reaction-diffusion system in any dimensions.
Moreover, depending on the given integral kernel, we demonstrate that the parameters of the approximating reaction-diffusion system can be explicitly determined up to three dimensions.
The pivotal idea of this result is that any radial integral kernel in $L^1( \R^n)$ space can be approximated by a superposition of the Green functions corresponding to the quasi-steady state of an auxiliary reaction-diffusion equation.
Although each Green function in high-dimensional cases has singularity at the origin, it can be demonstrated that its linear sum can approximate radial continuous functions.
By using this result, we will specify the mathematical relationship between nonlocal interactions and multiple chemical diffusive substances in dynamical systems in high-dimensional spaces.

This paper is organized as follows.
In Section \ref{sec:ms_r}, we explain the mathematical settings and state our main results. 
In Section \ref{sec:rda}, we prove the singular limit of the reaction diffusion system.
In Section \ref{sec:ker}, we give the proof for the result of the Green function expansion in any dimensions up to three.
Thereafter, we perform the numerical simulations of the Green function expansions in Section \ref{sec:main}.
We give some remarks and conclude this paper in Section \ref{sec:con}.

\section{Mathematical setting and main results}\label{sec:ms_r}
Based on the motivation introduced in Section \ref{sec:intro},
we consider the following nonlocal reaction-diffusion equation:
\be\label{pro:non}
\begin{cases}\tag{NP}
\dfrac{\partial u}{\partial t} = D\Delta u + f(u,K*u),  & (t>0,\ x\in\bR^n), \vspace{2mm}\\
u(0,x)=u_0(x), &(x\in\bR^n),
\end{cases}
\ee
where $u= u(t,x)\in\bR\ (t>0,\ x\in\bR^n, \ {\red n \in \bN})$, $D$ is a positive constant, 
the function $f\in \mathrm{Lip}(\bR^2;\bR)$ satisfies $f(0,0)=0$, 
and $K(x)$ is a radial function in $L^1(\bR^n)$.
Here, there exists a constant $C_f>0$ such that
\beaa
|f(u_1,v_1) - f(u_2,v_2)| \le C_f ( |u_1-u_2| + |v_1-v_2|)
\eeaa
for all $(u_1,v_1), (u_2,v_2)\in\bR^2$.
For the case that $f$ is the local Lipschitz as introduced in models in Section \ref{sec:intro}, we provide Remark \ref{rem:Lip} below. 

We first consider the existence of the solution to \eqref{pro:non}.
Denote the set of all bounded continuous functions by $BC(\bR^n)$.
Introducing the heat kernel as 
\[
H(t,x;D):= \dfrac{1}{(4\pi D t)^{n/2}} e^{-|x|^2/4Dt},
\]
we define the following map  as
\begin{equation*}
    {\red \cP[\phi](t)}:= H(t;D)*u_0 + \int^{t}_{0} H(t-s;D)* f(\phi(s),(K*\phi)(s)) ds 
\end{equation*}
for $\phi \in C([0,T];BC(\R^n) \cap L^p(\R^n))$ with $1 \le p \le +\infty$.
We say that a function $u \in C([0,T];BC(\R^n) \cap L^p(\R^n))$ for $T>0$ is a mild solution to \eqref{pro:non}, provided $u = \cP[u]$.
Then we obtain the following existence result.
\begin{theorem}\label{thm:exi}
For any $T>0$ and $1\le p\le +\infty$, there exists a unique mild solution $u\in C([0,T];BC(\bR^n)\cap L^p(\bR^n))$ to \eqref{pro:non} with an initial datum $u_0\in BC(\bR^n)\cap L^{p}(\bR^n)$.
This mild solution $u$ belongs to $ C^{1,2}((0,T]\times \bR^n)$, 
that is, $u$ is the unique classical solution to \eqref{pro:non}.
Moreover, we have
\beaa
  \| u (t) \|_{L^q}  \le e^{C_f (1+ \|K\|_{L^1}) t} \|u_0\|_{L^q}
\eeaa
for any $p\le q\le +\infty$ and $t \in [0,T]$.
\end{theorem}
We construct the mild solution based on the fixed point theorem for integral equations.
The proof is a standard argument and is summarized in Appendix \ref{app:non}.

Next, to specify the relationship between nonlocal interactions and reaction-diffusion system, we prepare a reaction–diffusion system for the approximation of the solution to \eqref{pro:non} with any integral kernel.
Introducing the auxiliary diffusive substances $v_j={\red v_j(t,x)}, \ (j=1,\ldots,N), \ (N\in\bN)$, we consider the following reaction-diffusion system with $N + 1$ components:
\begin{equation}\label{eq:rd}
\begin{cases}\tag{\ensuremath{\mbox{RD}_{\delta}}}
\dfrac{\partial u}{\partial t} = D\Delta u + f\left(u,\ \dsum^{N}_{j=1} \alpha_j v_j \right), \\
\dfrac{\partial v_j}{\partial t} = \dfrac{1}{\delta} (d_j \Delta v_j  - v_j + u), \quad (j=1,2,\ldots,N),
\end{cases}
(t>0,\ x\in\bR^n),
\end{equation}
where $\delta>0$, $\alpha_j \in\bR,\ d_j>0$ for $j=1,2,\ldots N$.
The idea using the reaction-diffusion equations for approximating  nonlocal interactions is same as that in \cite{NTY1,NTY2}.
The initial condition is imposed as
\be\label{init}
(u, v_1,\ldots, v_N)(0,x) = (u_0,k_1*u_0,\ldots, k_N*u_0)(x),
\ee
where $k_j,(j=1,\ldots,N)$ are defined as follows:
\bea
k_j(x) &:=& \left(\dfrac{1}{d_j}\right)^{n/2} G\left( \dfrac{|x|}{\sqrt{d_j}} \right),  \label{kj}\\
G(|x|) &:=& \dfrac{1}{(2\pi)^{n/2}}\left( \dfrac{1}{|x|}\right)^{n/2-1} M_{n/2-1}\left(|x|\right), \label{green}\\
M_{\nu}(r) &:=& \int^{+\infty}_{0} e^{-r \cosh s} \cosh(\nu s)ds.\notag
\eea
Here $M_{\nu}(r)$ is the modified Bessel function of the second kind with the order $\nu$.
For $j\in\bN$, $k_j$ is represented as
\begin{equation*}
    k_j (x) = 
\begin{cases}
    \dfrac{1}{2\sqrt{d_j}} e^{-|x|/\sqrt{d_j}}, &(n=1), \vspace{2mm} \\
    \dfrac 1 {2\pi d_j} M_0(|x|/\sqrt{d_j}), &(n=2), \vspace{2mm} \\
    \dfrac{1}{4\pi d_j |x|} e^{-|x|/\sqrt{d_j}}, &(n=3).
\end{cases}
\end{equation*}
We note that $k_j$ is the Green function of the differential operator $-d_j \Delta + 1$, that is, $w = k_j*u$ satisfies
\beaa
d_j \Delta w -w + u =0.
\eeaa
Moreover, it is known that $k_j \in L^1(\bR^n)$.
The property we use are mentioned in Subsection \ref{subsec:Green}.

As a preparation of the approximation, we show the existence of the solution to \eqref{eq:rd} with an initial condition \eqref{init}.
Define the two following maps as
\begin{align*}
&{\red \Phi[\phi](t)} := H(t;D)*u_0 + \int^{t}_{0} H(t-s;D)* f\left(\phi(s), \dsum^{N}_{j=1} \alpha_j \Psi_j[\phi](s)\right) ds, \\
&{\red \Psi_j[\phi](t)} :=  e^{-t/\delta} H(t;d_j/\delta)*(k_j*u_0) \\
& \qquad \qquad \qquad + \dfrac{1}{\delta} \int^{t}_{0} e^{-(t-s)/\delta} H(t-s;d_j/\delta)*\phi(s) ds,\quad (j=1,2,\ldots,N)
\end{align*}
for $\phi \in C([0,T];BC(\bR^n)\cap L^p(\bR^n))$ with $1 \le p \le +\infty$, respectively.
We say that a function $(u, v_1,\ldots,\ v_{N}) \in (C([0,T];BC(\bR^n)\cap L^p(\bR^n)))^{N+1}$ for $T>0$ is a mild solution to \eqref{eq:rd}, provided $u = \Phi[u]$ and $v_j = \Psi_j[u]$, respectively.
By using the Banach fixed point theorem, we obtain the existence result for \eqref{eq:rd}.
\begin{proposition}\label{pro:exi-app}
    For any $T>0$, $\delta>0$ and $1\le p\le+\infty$, there exists a unique mild solution $(u^{\delta},v^{\delta}_1,\ldots,v^{\delta}_N)\in \{C([0,T];BC(\bR^n)\cap L^p(\bR^n))\}^{N+1}$ to \eqref{eq:rd} with an initial condition \eqref{init} and $u_0\in BC(\bR^n)\cap L^p(\bR^n)$. 
    This solution belongs to $\{C^{1,2}((0,T]\times\bR^n)\}^{N+1}$.
    Moreover, we have 
    \beaa
    \| u^{\delta}(t) \|_{L^q}  &\le& \left(1 + \delta C_f \dsum^{N}_{j=1}|\alpha_j|  \right) \|u_0\|_{L^q} \exp\left( C_f \left(1 +  \dsum^{N}_{j=1}|\alpha_j|  \right) t \right),\\
    \| v^{\delta}_{j}(t) \|_{L^q}  &\le& e^{-t/\delta}\|u_0\|_{L^q} +\sup_{0\le s\le t} \| u^{\delta}(s)\|_{L^q}, \quad(j=1,2,\ldots,N)
    \eeaa
    for any $q\in [p,+\infty]$ and $t\in [0,T]$.
\end{proposition}
The proof is similar to the proof of Theorem \ref{thm:exi} and its main part about the boundedness is described in Appendix \ref{app:rd}.

For constants $\{\alpha_j\}_{1\le j\le N}$, 
and positive constants $\{d_j\}_{1\le j \le N}$, we set the linear combination of the Green functions as
\be\label{def:kn}
K_N(x):=\dsum^{N}_{j=1} \alpha_j k_j(x).
\ee
For the approximation of nonlocal interactions by the reaction-diffusion system
we show the singular limit as $\delta\to +0$ in \eqref{eq:rd} as follows.
\begin{proposition}\label{lem:err2}
    Let $T>0$, $\delta>0$, $1\le p\le +\infty$ and $u_0\in BC(\bR^n)\cap L^p(\bR^n)$.
    Suppose that $(u^{\delta},v^{\delta}_1,\ldots,v^{\delta}_N)\in \{C([0,T];BC(\bR^n)\cap L^p(\bR^n))\}^{N+1}$ is the solution to \eqref{eq:rd} with an initial condition \eqref{init}.
    Then, there exists $C_{1}=C_{1}(f,D,\{\alpha_j\}_{1\le j\le N},\{d_j\}_{1\le j\le N},T)$ such that
    \beaa
    \sup_{0\le t\le T} \|u^{\delta}(t)-u^{0}(t)\|_{L^p} &\le& C_{1}\|u_0\|_{L^p}\delta, \\
    \sup_{0\le t\le T} \|v^{\delta}_{j}(t)-k_j*u^{0}(t)\|_{L^p} &\le& C_{1}\|u_0\|_{L^p}\delta    
    \eeaa
    hold, where $u^{0}$ is the solution to \eqref{pro:non} with $K=K_N$.
\end{proposition}
This proposition shows the relationship between the solutions to \eqref{eq:rd} and \eqref{pro:non} with $K_N$.
Hence, it indicates that the solution to \eqref{eq:rd} is sufficiently close to that of \eqref{pro:non} with $K_N$.
Next, to approximate the solution to \eqref{pro:non} with any integral kernel by that to \eqref{eq:rd}, we prepare the following lemma.
\begin{lemma}\label{lem:err1}
    For $1\le p\le +\infty$, let $u$ be the solution to \eqref{pro:non} with
    an initial datum $u_0 \in BC(\bR^n)\cap L^p(\bR^n)$.
    Then, for any $T>0$, there exists $C_{2}=C_{2}(f,K,T)>0$ such that
    \beaa
    \sup_{0\le t\le T}\|u(t) - u^{0}(t) \|_{L^p} \le C_{2} e^{C_f \|K-K_N\|_{L^1}T}\|K-K_N\|_{L^1} \|u_0\|_{L^{p}} 
    \eeaa
    holds, where $u^{0}\in C([0,T];BC(\bR^n)\cap L^1(\bR^n))$ is the solution to \eqref{pro:non} with $K=K_N$.
\end{lemma}
This implies that the difference of the solutions to \eqref{pro:non} with $K$ and $K_N$ can be bounded by the difference of $K$ and $K_N$ in $L^1(\R^n)$ space.
{\red We thus identify the class of kernels that can be approximated by $K_N$ in the sense of the $L^1$-norm. First, we have the following as a general result.
\begin{theorem}\label{thm:ker-gen}
    Let $K\in L^1(\bR^n)$ be a radial function.
    Then, for any $\eps>0$, there exist $N\in\bN$, constants $\{\alpha_j\}_{1\le j \le N}$ and positive constants $\{d_j\}_{1\le j \le N} $ such that
    \beaa
    \|K - K_N\|_{L^1} <\eps
    \eeaa
    holds, where $K_N$ is defined in \eqref{def:kn}.
\end{theorem}
This ensures that we can approximate any radially symmetric kernel $K$ belonging to $L^1(\bR^n)$ with $K_N$ to arbitrary accuracy.
However, as shown in the proof in Section \ref{sec:ker}, 
$\{\alpha_j\}$ and $\{d_j\}$ are not explicitly determined, and it is not clear how they should be chosen to obtain an error estimate that depends on $N$.

Here, in order to {\blue seek the parameters $\{\alpha_j, d_j \}$}, 
we introduce the following assumptions for $K$ depending on the dimension up to three.
}


\begin{assumption}\label{ass:kernel}
    Let $K\in L^1(\bR^n)$ be a radial function. Let us denote $K(x)=J(|x|)$.
\begin{itemize}
    \item When $n=1$, we assume that $J\in C([0,+\infty))$ and that 
    $\dlim_{r\to+\infty} e^{r}J(r)$ exists. Then, for $\lambda \in (0,1]$ we set 
    \[
    h(\lambda):= \dfrac{J(-\log \lambda)}{\lambda}, 
    \]
    where we define $h(0):= \dlim_{r\to+\infty} e^{r}J(r)$;
    \item When $n=2$, we assume that $J\in  C^1([0,+\infty))$, and that there exists $\alpha>1/2$ such that 
    \[ 
    \dlim_{r\to+\infty} r^\alpha e^{r}J'(r)
    \] exists. Then, introducing the following function 
    \beaa
    A(r) := -\dfrac{2r}{\pi} \int^{+\infty}_{0} J'(r \cosh s)  ds, 
    \eeaa
    we set for $\lambda \in (0,1]$, 
    \beaa
    h(\lambda):= \dfrac{A(-\log \lambda)}{\lambda}.
    \eeaa 
    Here $h(0):=\dlim_{r\to+\infty} e^{r}A(r)=0$;
    \item When $n=3$, we assume that $J\in C((0,+\infty))$ and that both 
    $\dlim_{r\to+0} rJ(r)$ and $\dlim_{r\to+\infty} re^{r}J(r)$
    exist. Then, for $\lambda \in (0,1)$ we set
    \beaa
    h(\lambda):= -(\log\lambda )\dfrac{J(-\log \lambda)}{\lambda}, 
    \eeaa 
    where we define $h(0):= \dlim_{r\to+\infty} re^{r}J(r)$ and $h(1):= \dlim_{r\to+0} rJ(r)$.
\end{itemize}
\end{assumption}

Using this assumption, we can estimate the error.
The following error estimate provided by an expansion of the Green function $k_j$ is one of our main results.
\begin{theorem}\label{thm:est_poly}
Let $n\in\{1,2,3\}$ and $d_j = j^{-2}$ for $j\in\bN$. 
Let Assumption \ref{ass:kernel} be enforced.
Then, for any $N\in\bN$ and constants $\{\alpha_j\}_{1\le j \le N+1}$, 
\beaa
\left\|K - K_{N+1}  \right\|_{L^1} \le 
\dfrac{2\pi^{n/2}}{\Gamma(n/2)} \max_{\lambda\in[0,1]} | h(\lambda) - P_{N}(\lambda) |
\eeaa
holds, where $P_{N}(\lambda)$ is the polynomial defined as
\beaa
P_N(\lambda) := \dsum^{N}_{j=0} \alpha_{j+1} c_{j,n} \lambda^{j}, \quad
c_{j,n} :=
\begin{cases}
    \dfrac{j+1}{2}, &(n=1), \vspace{2mm} \\
    \dfrac{(j+1)^2}{2\pi}, &(n=2), \vspace{2mm} \\
    \dfrac{(j+1)^2}{4\pi}, &(n=3).
\end{cases}
\eeaa
\end{theorem}

This theorem shows that the approximation error between $K$ and $K_{N+1}$ is evaluated in terms of the absolute error between the function $h$ determined from the integral kernel and the polynomial $P_N$. 
This allows us to examine how to determine $\{\alpha_j\}_{1\le j\le N+1}$ based on polynomial approximation theory.
Since the properties of the Green function depend on spatial dimensions,
we need Assumption \ref{ass:kernel} for the  approximation by $K_{N+1}$.
In the cases that $n=1$ or $n=3$, $k_j$ satisfies this assumption for all  $0<d_j\le 1$.
On the other hand, when $n=2$, $k_j$ does not satisfy the assumption for any $d_j>0$.
In this sense, the case that $n=2$ is a technical assumption.
The proof is given in Section \ref{sec:ker}.

We provide examples of how to choose the coefficients $\{\alpha_j\}_{1\le j \le N+1}$. 
Before stating the result, we prepare some notions.
Let $h\in C([0,1])$.
We define the modulus of continuity of $h$ on $[0,1]$ as
\beaa
\omega(h,\eta):= \sup\left\{ |h(\lambda_1)-h(\lambda_2)|\ ;\ \lambda_1,\lambda_2\in [0,1],\ |\lambda_1-\lambda_2|\le \eta\right\}.
\eeaa
For $N\in\bN$, let us define constants $\beta_{j,N}[h]\ (0\le j\le N)$ as
\be\label{coef}
\beta_{j,N}[h]:=\dsum^{j}_{\nu=0} (-1)^{j-\nu}  h\left( \dfrac{\nu}{N} \right)  \dbinom{N}{j}\dbinom{j}{\nu},\quad (j=0,1,\ldots,N),
\ee
where $\binom{\cdot}{\cdot}$ is the binomial coefficient.
The constant $\beta_{j,N}[h]$ plays the role to determine the coefficients of the Bernstein polynomial below.
Furthermore, we set the following constant
\[
l_{j,N} [h]:= \frac{1}{2^{2N+1}}  \sum_{ \nu=0}^N \zeta_{\nu,N} (h) \tau_{j,\nu}^{(N+1)}, \quad (j=0,1,\ldots,N),
\]
where $\zeta_{\nu,N} (h)$ and $\tau_{j,\nu}^{(N+1)}$ are defined in Subsection \ref{sec:Lag}.
The constant $l_{j,N}[h]$ plays the role to determine the coefficients of the Lagrange polynomial with the Chebyshev nodes below.
Using these coefficients of the polynomials, we obtain the following explicit estimates.

\begin{corollary}\label{cor:ker}
Let $n\in\{1,2,3\}$ and $d_j = j^{-2}$ for $j\in\bN$. 
Let Assumption \ref{ass:kernel} be enforced.
\begin{itemize}
\item When $\alpha_j = \beta_{j-1,N}[h]/c_{j-1,n}$, for $m\in\{0,1\}$, there exists a constant $E(m)>0$ independent of $K$ such that if $h\in C^{m}([0,1])$ holds,
then 
\beaa
\left\|K - K_{N+1}  \right\|_{L^1} \le 
\dfrac{2E(m)\pi^{n/2}}{\Gamma(n/2)} N^{-m/2} \omega(h^{(m)},N^{-1/2}).
\eeaa
\item When $\alpha_j = l_{j-1,N}[h]/c_{j-1,n}$, for  $ h\in\mathrm{Lip}([0,1])$, it holds that 
\beaa
\left\|K - K_{N+1}  \right\|_{L^1} \le 
\dfrac{2 \pi^{n/2}}{\Gamma(n/2)} \Big(2+\frac{2}{\pi}\log N\Big)  \omega ( h, N^{-1} ).
\eeaa
\end{itemize}
\end{corollary}
\begin{remark}\label{rem:Cm}
 As in Lemma \ref{lem:TLag_sm}, for $h \in C^{m}[0,1]$ with $m \in \bN$ in the case of $\alpha_j = l_{j-1,N}[h]/c_{j-1,n}$, the convergence order becomes $O(N^{-m})$.
\end{remark}
This corollary shows that any nonlocal interactions can be approximated by a linear sum of $k_j$ in high-dimensional Euclidean spaces and also the convergence rate of the error can be specified.
The proof of Corollary \ref{cor:ker} is in Section \ref{sec:ker}.

Under the above mentioned settings, we obtain the main approximation result to the nonlocal problems \eqref{pro:non} by using the reaction-diffusion system \eqref{eq:rd} as follows.
\begin{theorem}\label{thm:rda}
{\red For any {\blue $n \in \bN$,} radial integral kernel $K\in L^1(\bR^n)$,}
$\eps>0$, $\delta>0$ and $T>0$, there exist $N =N(K,n,\eps)\in \bN$,  a reaction diffusion system \eqref{eq:rd} with $N+1$ components, and positive constants 
\beaa
C_{1}=C_{1}(f,D,\{\alpha_j\}_{1\le j\le N},\{d_j\}_{1\le j\le N},T),\quad C_{3}=C_{3}(f,K,T),\quad
\eeaa
such that for any $1\le p \le +\infty$,
\beaa
\sup_{0\le t\le T}\| u^{\delta}(t) - u(t)\|_{L^{p}} &\le& (C_1\delta +C_3 \eps ) \|u_0\|_{L^{p}}, \\
\sup_{0\le t\le T}\| v^{\delta}_j(t)  - (k_j*u)(t)\|_{L^{p}} &<& (C_1\delta + C_3\eps ) \|u_0\|_{L^{p}},\quad (j=1,2,\ldots,N),
\eeaa
where 
$u$ is the solution to \eqref{pro:non} with $u_0 \in BC(\bR^n)\cap L^{p}(\bR^n)$, and $(u^{\delta},v^{\delta}_{1},\ldots,v^{\delta}_{N})$ is the solution to \eqref{eq:rd} with \eqref{init}.
\end{theorem}
This theorem shows that solutions to \eqref{pro:non} with any integral kernel in high-dimensional spaces can be realized by the solution to multiple component reaction-diffusion system in high-dimensional spaces.

\begin{remark}
    We can take the limit of $\delta \to +0$ in \eqref{eq:rd} since the constants $C_1$ and $C_3$ are independent of $\delta$. 
    Furthermore, after this limit, we can take the limit of  $\eps \to +0$ in $(\textup{RD}_0)$ since the constant $C_3$ is independent of $N$. 
    If Lemma \ref{lem:poly} in Subsection \ref{sec:BP} or Lemma \ref{lemm:Lag} in Subsection \ref{sec:Lag} is applied,
    Theorem \ref{thm:rda} also shows the convergence rate with respect to $\delta$ and $N$.
\end{remark}
\begin{remark}\label{rem:Lip}
    The global Lipschitz condition on $f(u,v)$ can be removed if the boundedness of the solutions can be guaranteed.
    When $f(u,v)$ is a locally Lipschitz continuous function on $\bR^2$ and $u_0\in BC(\bR^n)\cap L^{p}(\bR^n)$ for $1\le p\le +\infty$, 
    the same assertion as Theorem \ref{thm:rda} follows by choosing a sufficiently large $C_f>0$ if 
    \beaa
    \sup_{0\le t\le T} (\|u(t)\|_{L^\infty} + \|u^{0}(t)\|_{L^{\infty}} + \|u^{\delta}(t)\|_{L^{\infty}}) <+\infty
    \eeaa
    is obtained a priori for some $T>0$ and $\delta>0$.
    Here, $u^{0}$ is a solution to \eqref{pro:non} with $K=K_N$.
\end{remark}
\begin{remark}
    When $p=+\infty$ and $u_0\in BC(\bR^n)$ is a periodic function, 
    the solution is also a spatially periodic function with the same period. 
    Therefore, the approximation in Theorem \ref{thm:rda} can be used to evaluate solutions with periodic boundary conditions.
\end{remark}

The necessary lemmas for proof of Theorem \ref{thm:rda} is given in Section \ref{sec:rda} and the proof of Theorem \ref{thm:rda} is given in Section \ref{sec:main}.

\section{Error estimates}\label{sec:rda}

\subsection{Properties of the Green function}\label{subsec:Green}

Here we describe some properties of the Green function that we use.
    Some properties are also described in \cite{IT}, 
but we summarize those necessary properties for the ease of the reader.

The Fourier transform of $G(|x|)$ defined by \eqref{green} is represented by
\beaa
\hat{G}(\xi) = \mathcal{F}_n[G](\xi) &:=& \int_{\bR^n} e^{-ix\cdot \xi} G(|x|)dx \\
&=&  (2\pi)^{n/2} \left( \dfrac{1}{|\xi|}\right)^{n/2-1} \int^{+\infty}_{0} r^{n/2} J_{n/2-1}(|\xi|r) G(r)dr \\
&=&\left( \dfrac{1}{|\xi|}\right)^{n/2-1} \int^{+\infty}_{0} r J_{n/2-1}(|\xi|r) M_{n/2-1}\left(r\right)dr
\eeaa
for $|\xi|>0$,
where $J_{\nu}(r)$ is the Bessel function of the first kind with the order $\nu$.
Using formula
\beaa
\int^{+\infty}_{0} r J_{n/2-1}(|\xi|r) M_{n/2-1}\left(r\right)dr= |\xi|^{n/2-1} {}_2 F_1 \left( \dfrac{n}{2}, 1; \dfrac{n}{2} ; - |\xi|^2 \right) = \dfrac{|\xi|^{n/2-1}}{1+|\xi|^2}
\eeaa
from \S 13.45 in \cite{Watson},
we obtain $\hat{G}(\xi) = \dfrac{1}{1+|\xi|^2}$.
Here, ${}_2F_1(a,b; c ; z)$ is the hypergeometric function.
It is obvious that $G(|x|)$ is positive for $|x|>0$.
Moreover, we find that
\begin{align*}
\int_{\bR^{n}} G(|x|) dx 
&= \dfrac{2\pi^{n/2}}{\Gamma(n/2)} \int^{+\infty}_{0} r^{n-1} G(r) dr \\
&= \dfrac{1}{2^{n/2-1}\Gamma(n/2)} \int^{+\infty}_{0} r^{n/2} M_{n/2-1}(r) dr=1    
\end{align*}
by using the integral formula
\be\label{integ-form}
\int^{+\infty}_{0} r^{\mu-1} M_{\nu}(r)dr = 2^{\mu-2} \Gamma\left(\dfrac{\mu-\nu}{2}\right) \Gamma\left(\dfrac{\mu+\nu}{2}\right)
\ee
with $|\mathrm{Re}(\nu)|< \mathrm{Re}(\mu)$ from \S 13.21 in  \cite{Watson}.
In summary, $k_j$ defined in \eqref{kj} has the following properties:
\begin{lemma}\label{lem:green}
    For $j\in\bN$, we have
    \begin{itemize}
        \item[(i)] $k_j\in C(\bR^n\backslash\{0\})\cap L^1(\bR^n)$ and
        \beaa
        \mathcal{F}_n[k_j](s)= \dfrac{1}{1+d_j |\xi|^2};
        \eeaa
        \item[(ii)] $k_j(x)>0\ (x\neq 0)$ and $\|k_j\|_{L^1}=1$.
    \end{itemize}
\end{lemma}

In Section \ref{sec:ker}, we use the asymptotic properties of $M_{\nu}$.
From \cite{NIST}, it is known the asymptotic properties
\bea\label{BK:asy}
M_{\nu}(r) \simeq\sqrt{\dfrac{\pi}{2r}} e^{-r},\quad (r\to+\infty)
\eea
and
\beaa
M_{\nu}(r) \simeq 
\begin{cases}
\dfrac{\Gamma(\nu)}{2} \left(\dfrac{r}{2}\right)^{-\nu}, &(\nu>0),\\
-\log r,  &(\nu=0),
\end{cases}
\quad (r\to+0)
\eeaa
for any $\nu>0$, where $f(r)\simeq g(r) \ (r\to a)$ means
\beaa
\dlim_{r\to a} \dfrac{f(r)}{g(r)} = 1.
\eeaa

\subsection{Error estimates for reaction-diffusion approximation}
In this subsection, we consider the singular limit problem for solutions to \eqref{eq:rd} with an initial condition \eqref{init}.



\begin{proof}[Proof of Proposition \ref{lem:err2}]
    Let $w(t)= u^{\delta}(t)- u^{0}(t)$ and $z_j(t)= v^{\delta}_j(t)- (k_j*u^{0})(t)$.
    Then, we have
    \beaa
    \delta \dfrac{\partial z_j}{\partial t} = d_j \Delta z_j  -z_j + w - \delta k_j* \dfrac{\partial u^{0}}{\partial t}.
    \eeaa
    This implies that
    \beaa
    z_j (t) = \dfrac{1}{\delta} \int^{t}_{0} e^{-(t-s)/\delta} H(t-s;d_j/\delta)* \left[w(s) - \delta \left( k_j* \dfrac{\partial u^{0}}{\partial t}\right)(s)\right] ds
    \eeaa
    holds for any $j=1,2,\ldots, N$.
    Notice that
    \beaa
    k_j * \dfrac{\partial u^{0}}{\partial t}(t) &=& [k_j * (D\Delta u^{0} + f(u^{0},K_N*u^{0}))](t) \\
    &=& \dfrac{D}{d_j} u^{0}(t) + \dfrac{D}{d_j} ( k_j * u^{0})(t) + [k_j * f(u^{0},K_N*u^{0}))](t),
    \eeaa
    and then, we find that
    \be\label{ieq:time}
    \left\|\left( k_j* \dfrac{\partial u^{0}}{\partial t}\right)(t)\right\|_{L^p}
    \le  \left[ \dfrac{2D}{d_j} + C_{f} (1+\|K_N\|_{L^1}) \right] \| u^{0}(t)\|_{L^p}
    \ee
    for any $t \in (0,T]$ from the Young inequality and Lemma \ref{lem:green}. 

    Let $t\in (0,T]$. Then, we obtain that
    \be\label{ieq:zj}
    \| z_j(t) \|_{L^{p}} 
    \le \dfrac{1}{\delta} \int^{t}_{0} e^{-(t-s)/\delta} \left[ \|w(s)\|_{L^p} + \delta  \left\|\left( k_j* \dfrac{\partial u^{0}}{\partial t}\right)(s)\right\|_{L^p} \right] ds.
    \ee
    Moreover, we have
    \beaa
    \| w(t) \|_{L^{p}} 
    &\le& C_f  \int^{t}_{0} \left[ \|w(s)\|_{L^p} +  \dsum^{N}_{j=1} |\alpha_j| \|z_j(s)\|_{L^{p}} \right] ds \\
    &\le& C_f \int^{t}_{0} \|w(s)\|_{L^p}ds \\
    &&\quad + \dfrac{C_f}{\delta}  \dsum^{N}_{j=1} |\alpha_j| \int^{t}_0 \int^{s}_{0} e^{-(s-\eta)/\delta} \left[ \|w(\eta)\|_{L^p} + \delta  \left\|\left( k_j* \dfrac{\partial u^{0}}{\partial t}\right)(\eta)\right\|_{L^p} \right] d\eta ds \\
    &\le& C_f \int^{t}_{0} \|w(s)\|_{L^p}ds \\
    &&\quad + C_f  \dsum^{N}_{j=1} |\alpha_j| \int^{t}_0 (1- e^{-(t-s)/\delta}) \left[ \|w(s)\|_{L^p} + \delta  \left\|\left( k_j* \dfrac{\partial u^{0}}{\partial t}\right)(s)\right\|_{L^p} \right] ds \\
    &\le& C_f \left( 1 +  \dsum^{N}_{j=1} |\alpha_j| \right) \int^{t}_{0} \|w(s)\|_{L^p}ds 
    + \delta C_f \dsum^{N}_{j=1} |\alpha_j| \int^{t}_0 \left\|\left( k_j* \dfrac{\partial u^{0}}{\partial t}\right)(s)\right\|_{L^p} ds.
    \eeaa
    By using the Gronwall inequality, 
    we deduce
    \beaa
    \|w(t)\|_{L^p} &\le& \delta C_f \left( \dsum^{N}_{j=1} |\alpha_j| \int^{t}_0 \left\|\left( k_j* \dfrac{\partial u^{0}}{\partial t}\right)(s)\right\|_{L^p} ds \right) \exp\left( C_f \left( 1 +  \dsum^{N}_{j=1} |\alpha_j| \right) t \right). 
    \eeaa
    From \eqref{ieq:time} and Theorem \ref{thm:exi}, there exists 
    \[
    C_{11}=C_{11}(f,D,\{\alpha_j\}_{1\le j\le N},\{d_j\}_{1\le j\le N},T)>0
    \]
    such that
    \beaa
    \|w(t)\|_{L^p} \le C_{11}\|u_0\|_{L^p} \delta
    \eeaa
    holds. 
    From the similar argument for $w$ and \eqref{ieq:zj},
    there exists a positive constant $C_{12}=C_{12}(f,D,\{\alpha_j\}_{1\le j\le N},\{d_j\}_{1\le j\le N},T)$ such that
    \beaa
    \|z_j(t)\|_{L^p} \le C_{12}\|u_0\|_{L^p} \delta
    \eeaa
    holds. Thus, we obtain the desired assertion. 
\end{proof}

Next, we evaluate the continuity of the solution to \eqref{pro:non} with respect to the integral kernel in Lemma \ref{lem:err1}.

\begin{proof}[Proof of  Lemma \ref{lem:err1}]
    Let $1\le p \le +\infty$ and $t\in (0,T]$.
    Let $u_{\text{err}}(t) := u(t) - u^{0}(t)$ and $K_{\text{err}}(x):= K(x) -K_{N}(x)$.
    Then, we have
    \beaa
    |u_{\text{err}}(t)| &\le& C_{f} \int^{t}_{0} H(t-s;D)* (|u_{\text{err}}(s)| + |(K* u)(s) - (K_N*u^{0})(s)|ds \\
    &\le& C_f \int^{t}_{0} H(t-s;D)* (|u_{\text{err}}(s)| + |(K_{N}* u_{\text{err}})(s)| + |({K_{\text{err}}}*u)(s)|) ds.
    \eeaa
    From the Young inequality, we deduce
    \beaa
    \|u_{\text{err}}(t)\|_{L^{p}} &\le& C_{f} \int^{t}_{0} [ (1+\|K_N\|_{L^1}) \|u_{\text{err}}(s)\|_{L^p} + \|K_{\text{err}}\|_{L^1}\|u(s)\|_{L^{p}}] ds.
    \eeaa
    Using the Gronwall inequality and Theorem \ref{thm:exi} yields that
    \beaa
    \|u_{\text{err}}(t)\|_{L^{p}} &\le& C_f \left( \int^{t}_{0} \|u(s)\|_{L^{p}} ds\right) e^{C_f (1+\|K_N\|_{L^1})t} \|K_{\text{err}}\|_{L^1} \\
    &\le & C_f \left( \int^{t}_{0} e^{C_f (1+\|K\|_{L^1}) t} ds\right) e^{C_f (1+\|K_N\|_{L^1})t} \|K_{\text{err}}\|_{L^1} \|u_0\|_{L^p} \\
    &\le&  \dfrac{1}{1+\|K\|_{L^1}} e^{C_f (2+\|K\|_{L^1} +\|K_{N}\|_{L^1}) t} \|K_{\text{err}}\|_{L^1} \|u_0\|_{L^p} \\
    &\le&  \dfrac{1}{1+\|K\|_{L^1}} e^{C_f (2+2\|K\|_{L^1}+\|K_{\text{err}}\|_{L^1}) t} \|K_{\text{err}}\|_{L^1} \|u_0\|_{L^p}.
    \eeaa
    Thus, we obtain the desired assertion.
\end{proof}

\section{Approximation of a kernel by the Green function in $L^1(\bR^n)$}\label{sec:ker}

{\red Throughout of this section, we use the notation $K_N$ defined in \eqref{def:kn}.}


{\red 
\subsection{Kernel approximation in general dimension}

Let us explain the proof of Theorem \ref{thm:ker-gen}.
We apply the Wiener's approximation theorem:
\begin{theorem}[\cite{Wiener}]\label{thm:wiener}
    Let $F\in L^1(\bR)$. The following statements are equivalent:
    \begin{enumerate}
        \item $ \mathrm{span}\{F(s-\theta)\mid \theta\in\bR\}$ is dense in $L^1(\bR)$, i.e., for any $G\in L^1(\bR)$ and $\eps>0$,
        there exist $N\in\bN$, $\{\alpha_j\}_{1\le j\le N},\ \{\theta_j\}_{1\le j\le N} \subset\bR$ such that
        \beaa
        \left\|G - \dsum^{N}_{j=1} \alpha_j F(\cdot-\theta_j)\right\|_{L^1(\bR)}<\eps;
        \eeaa
        
        \item for any $\xi\in\bR$, we have
        \beaa
        \int_{\bR} e^{is\xi} F(s)ds \neq 0.
        \eeaa
    \end{enumerate}
\end{theorem}

This theorem provides necessary and sufficient conditions for approximating arbitrary functions in $L^1(\bR)$ by translations of functions. 
Based on this assertion, we show that $K$ can be approximated by $K_N$ on $L^1(\bR^n)$.

\begin{proof}[Proof of Theorem \ref{thm:ker-gen}]
    Let $K\in L^1(\bR^n)$ be a radial function.
    Write it as $K(x)=J(|x|)$.
    We define
    \beaa
    \cH[K](s) := e^{n s} J(e^{s}) \quad (s\in\bR).
    \eeaa
    Then, we have
    \beaa
    \| K \|_{L^1(\bR^n)} &=& |\bS^{n-1}| \int^{+\infty}_{0} r^{n-1} |J(r)|dr \\
    &=& |\bS^{n-1}| \int_{\bR} |\cH[J](s)|ds = |\bS^{n-1}| \| \cH[K] \|_{L^1(\bR)},
    \eeaa
    where $|\bS^{n-1}|$ is the area of the $(n-1)$-dimensional unit sphere.
    Set 
    \beaa
    k(x;d) := \left(\dfrac{1}{d}\right)^{n/2} G\left( \dfrac{|x|}{\sqrt{d}} \right)
    \eeaa
    as the function obtained by replacing $d_j$ in \eqref{kj} with $d>0$.
    Here, putting $ \theta := (\log d )/2$, we obtain
    \be\label{trans}
    \cH[k(\cdot;d)](s) = e^{n(s-\theta)} G(e^{s-\theta}) = \cH[k(\cdot;1)](s-\theta).
    \ee

    Let us show that the Fourier transform of $\cH[k(\cdot;1)]$ does not attain $0$.
    Considering the variable transformation $\eta = e^{s}$, for any $\xi \in \bR$, we deduce
    \beaa
    \int_{\bR} e^{i\xi s} \cH[k(\cdot;1)](s)ds 
    &=& \int_{\bR} e^{i\xi s} e^{ns} G(e^{s}) ds \\
    &=& \dfrac{1}{(2\pi)^{n/2}} \int_{\bR} e^{i\xi s} e^{(n/2+1)s} M_{n/2-1}(e^{s}) ds \\
    &=& \dfrac{1}{(2\pi)^{n/2}} \int^{+\infty}_{0} \eta^{i\xi+n/2} M_{n/2-1}(\eta) d\eta.
    \eeaa
    It follows from \eqref{integ-form} that
    \beaa
    \int_{\bR} e^{i\xi s} \cH[k(\cdot;1)](s)ds 
    = \dfrac{1}{\pi^{n/2}} 2^{-1 +i\xi} \Gamma\left(\dfrac{2+i\xi}{2}\right) \Gamma\left(\dfrac{n+i\xi}{2}\right) \neq 0.
    \eeaa

    From Theorem \ref{thm:wiener}, for any $\eps>0$, there exist $N\in\bN$, $\{\alpha_j\}_{1\le j\le N},\ \{\theta_j\}_{1\le j\le N} \subset\bR$ such that
    \beaa
    \left\| \cH[K] - \dsum^{N}_{j=1} \alpha_j \cH[k(\cdot;1)](\cdot- \theta_j) \right\|_{L^1(\bR)} < \dfrac{\eps}{|\bS^{n-1}|}.
    \eeaa
    Therefore, by setting $d_j = e^{2\theta_j}\ (1\le j\le N)$, we obtain
    \beaa
    \left\|K - K_N \right\|_{L^1(\bR^n)} 
    &=& |\bS^{n-1}| \left\| \cH[K] - \dsum^{N}_{j=1} \alpha_j \cH[k(\cdot;d_j)] \right\|_{L^1(\bR)} \\
    &=& |\bS^{n-1}| \left\| \cH[K] - \dsum^{N}_{j=1} \alpha_j \cH[k(\cdot;1)](\cdot- \theta_j) \right\|_{L^1(\bR)} < \eps.
    \eeaa
    This completes the proof.
\end{proof}

}

\subsection{Error estimate for integral kernels}
We explain the proof of Theorem \ref{thm:est_poly} in this subsection.
{\red Throughout of this subsection, we set $d_j=j^{-2}$ for $j\in\bN$.}
For the result of {\red $L^1(\bR^n)$} convergence stated in Theorem \ref{thm:est_poly}, we can obtain the convergence result for derivatives.
Although this result is not used for the reaction-diffusion approximation of nonlocal interactions, we provide the following result on derivative approximation for Theorem \ref{thm:est_poly}.
\begin{theorem}\label{thm:w11}
Let $n\in\{1,2,3\}$ and $d_j = j^{-2}$ for $j\in\bN$. 
We assume the following conditions in each dimension:
\begin{itemize}
    \item When $n=1$, we assume that $J\in C^1([0,+\infty))$ and that {\red  $\dlim_{r\to+\infty} e^{2r}(J(r) + J'(r)) $ 
    exists.}
    \item When $n=2$, we assume that $J\in C^2([0,+\infty))$ 
    and that $\dlim_{r \to + 0} r^{-1} ( J'(r) + r J''(r) ) $ exists, and that there exists $\alpha>1/2$ such that
    \begin{equation*}
            {\red \dlim_{r\to+\infty} r^{\alpha} e^{2r} (J'(r) + r J''(r))}
    \end{equation*}
    exist.
    \item When $n=3$, we assume that $J\in C^1((0,+\infty))$ and that both  
    $\dlim_{r\to+0} ( J(r) + r J'(r))$ and {\red $\dlim_{r \to +\infty} e^{2r}( J(r) + rJ'(r) + rJ(r) )$} exist.
\end{itemize}
Let the same notation of $h$ in Assumption \ref{ass:kernel} be enforced.
Then, for any $N\in\bN$ and constants $\{\alpha_j\}_{1\le j \le N+1}$, 
\begin{align*}
\left\|K - K_{N+1}  \right\|_{W^{1,1}}
&\le \dfrac{2(n+1)\pi^{n/2}}{\Gamma(n/2)} \max_{\lambda\in[0,1]} | h(\lambda) - P_{N}(\lambda) |\\
&\quad\quad + \dfrac{2n\pi^{n/2}}{2^n\Gamma(n/2)} \max_{\lambda\in[0,1]} | h'(\lambda) - P'_{N}(\lambda) |    
\end{align*}
holds, where $P_{N}$ is the polynomial defined in Theorem \ref{thm:est_poly}. 
\end{theorem}
We provide the proofs of Theorem \ref{thm:est_poly} and Theorem \ref{thm:w11} in the successive sub-subsections, respectively.

\subsubsection{One-dimensional case}
We first give the proof of Theorem \ref{thm:est_poly} when $n=1$.
Note that in this case $k_j$ is represented as
\beaa
k_j (x) = \dfrac{j}{2} e^{-j|x|}.
\eeaa

\begin{proof}[Proof of Theorem \ref{thm:est_poly}]
Let Assumption \ref{ass:kernel} for $n=1$ be enforced.
Then, we obtain
\beaa
\left\|K - K_{N+1}  \right\|_{L^1} &=& 2 \int^{+\infty}_{0} \left| J(r) - \dsum^{N+1}_{j=1} \dfrac{j \alpha_j}{2} e^{-jr} \right| dr \\
&\le& 2 \left( \int^{+\infty}_{0} e^{-r} dr\right) \sup_{r\ge 0} \left|e^{r} J(r) - \dsum^{N}_{j=0} \alpha_{j+1} c_{j,n} e^{-jr} \right| \\
&=& 2\max_{\lambda\in[0,1]} \left|h(\lambda) - P_{N}(\lambda) \right|
\eeaa
from the definition of $\{\alpha_j\}_{1\le j\le N+1}$.
\end{proof}

From this proof, we obtain a point-wise absolute error as follows. 
\begin{corollary}
Let $n=1$ and $d_j = j^{-2}$ for $j\in\bN$. 
Let Assumption \ref{ass:kernel} be enforced.
Then, we have
\beaa
\left|K(x) - K_{N+1}(x) \right| \le  e^{-|x|} \max_{\lambda\in[0,1]} \left|h(\lambda) - P_{N}(\lambda) \right|
\eeaa
for any $x\in\bR$.
\end{corollary}

\begin{proof}[Proof of Theorem \ref{thm:w11}]
We note that \[
h'(\lambda) = - \frac{ J'(-\log \lambda) + J(-\log  \lambda) }{\lambda^2} = -e^{2r}( J'(r) + J(r) ), \ (-\log \lambda = r)
\]
and we define $h'(0) := \dlim_{r \to +\infty}  -e^{2r}( J(r) + J'(r) ) $.
Then, we compute that 
\begin{align*}
&\left\| \frac{ \partial }{ \partial x } (K - K_{N+1} )  \right\|_{L^1} 
=2 \int^{+\infty}_{0} \left| J'(r)  +  \sum_{j=1}^{N+1}  \frac{ j^2 \alpha_j }{2}  e^{-jr} \right| dr\\
&\le 2 \int^{+\infty}_{0} e^{-2 r} \left|  e^{2r} (J'(r) +J(r)) + \sum_{j=0}^N j \alpha_{j+1} c_{j,n} e^{-(j-1)r } \right| dr + \left\|K - K_{N+1}  \right\|_{L^1} \\
&\le 2 \Big( \int^{+\infty}_{0} e^{-2 r} dr \Big) \max_{\lambda \in [0,1] }  \left|  h'(\lambda) - P'_N (\lambda)  \right| + \left\|K - K_{N+1}  \right\|_{L^1} \\
&= \max_{\lambda \in [0,1] }  \left|  h'(\lambda) - P'_N (\lambda)  \right| + 2 \max_{\lambda \in [0,1] }  \left|  h(\lambda) - P_N (\lambda)  \right|.    
\end{align*}
Thus, the proof in one-dimensional case is complete.
\end{proof}

\subsubsection{Two-dimensional case}

We show the case that $n=2$ in Theorem \ref{thm:est_poly}.
Remark that $k_j$ is expressed as
\beaa
k_j(x) = \dfrac{j^2}{2\pi} M_{0}(j|x|) = \dfrac{j^2}{2\pi}  \int^{+\infty}_{0} e^{-j|x| \cosh s} ds.
\eeaa
Let Assumption $\ref{ass:kernel}$ for $n=2$ be enforced.
We prepare a lemma for $A(r)$.
The integral transformation $A(r)$ of $J(r)$ is defined as inspired by the Abel transformation (\cite{Beerends}).
Based on the theory, it follows 

\begin{lemma}\label{lemm:A}
    $A$ is well-defined on $[0,+\infty)$ and satisfies $A(0)=0$, $A\in C([0,+\infty))$ and 
     \beaa
      \dlim_{r\to+\infty} e^{r}A(r)=0.
     \eeaa
    Moreover, $J$ is represented by
    \beaa
    J(r) = \int^{+\infty}_{0} A(r \cosh s)ds. 
    \eeaa
\end{lemma}
\begin{proof}
    From Assumption \ref{ass:kernel},
    there exists a constant $C_J>0$ such that
    \beaa
    |J'(r)| \le C_J \min\{ 1, r^{-\alpha} \} e^{-r} \quad (r> 0).
    \eeaa
    The continuity of $A$ is obtained by the Lebesgue dominated theorem.
    For $r>0$, we obtain that
    \beaa
    |A(r)| 
    &\le& \dfrac{2r}{\pi} \int^{+\infty}_{0} |J'(r \cosh s)| ds \\
    &\le& \dfrac{2C_J r}{\pi} \int^{+\infty}_{0} \min\{1, (r\cosh s)^{-\alpha} \} e^{-r\cosh s} ds \\
    &\le&  \dfrac{2C_J } {\pi} \min\{r, r^{1-\alpha} \} M_0(r).
    \eeaa
    This implies that
     \beaa
      \dlim_{r\to+\infty} e^{r}A(r) =0
     \eeaa
    from \eqref{BK:asy}.
    Taking a limit as $r\to+0$ yields $A(0)=0$. 
    Finally, we have
\begin{align*}
    \int^{+\infty}_{0} A(r \cosh s)ds 
    &= \int^{+\infty}_{r} \dfrac{A(s)}{\sqrt{s^2-r^2}}  ds  \notag\\
    &= -\dfrac{2}{\pi} \int^{+\infty}_{r} \int^{+\infty}_{s} \dfrac{sJ'(\eta)}{\sqrt{s^2-r^2}\sqrt{\eta^2-s^2}} d\eta  ds \notag\\
    &= -\dfrac{2}{\pi} \int^{+\infty}_{r} \int^{\eta}_{r} \dfrac{s}{\sqrt{s^2-r^2}\sqrt{\eta^2-s^2}} ds J'(\eta) d\eta \notag\\
    &= -\int^{+\infty}_{r} J'(\eta) d\eta =J(r).   
\end{align*}
    Thus, we get the desired assertion.
\end{proof}

\begin{proof}[Proof of Theorem \ref{thm:est_poly}]
From the definitions of $h$ and $\{\alpha_j\}_{1\le j\le N+1}$,
we have
\begin{align*}
\left\|K - K_{N+1}  \right\|_{L^1} 
&= 2\pi \int^{+\infty}_{0} r \left| J(r) - \dsum^{N+1}_{j=1} \dfrac{j^2 \alpha_j}{2\pi} M_0(jr) \right| dr \\
&\le 2\pi \int^{+\infty}_{0} \int^{+\infty}_{0} r \left| A(r \cosh s) - \dsum^{N+1}_{j=1} \alpha_{j} c_{j-1,n} e^{-jr\cosh s} \right|ds dr \\
&\le 2\pi  \left( \int^{+\infty}_{0} \int^{+\infty}_{0} r e^{-r\cosh s} ds dr \right) \sup_{r\ge 0} \left| e^{r}A(r) - \dsum^{N}_{j=0} \alpha_{j+1} c_{j,n} e^{-jr} \right| \\
&= 2\pi \max_{\lambda\in[0,1]} \left|h(\lambda) - P_{N}(\lambda) \right|.
\end{align*}
\end{proof}

Similarly, we get a point-wise absolute error. 
\begin{corollary}
Let $n=2$ and $d_j = j^{-2}$ for $j\in\bN$. 
Let Assumption \ref{ass:kernel} be enforced.
Then, we have
\beaa
\left|K(x) - K_{N+1}(x) \right| \le  M_0(|x|) \max_{\lambda\in[0,1]} \left|h(\lambda) - P_{N}(\lambda) \right|
\eeaa
for all $x\in\bR^2\backslash\{0\}$.
\end{corollary}

\begin{proof}[Proof of Theorem \ref{thm:w11}]
First we show the following properties.
\begin{lemma}\label{lemm:Ad}
    $A'$ is well-defined on $[0,+\infty)$, $A\in C^1([0,+\infty))$, $A'(0)=0$ and $\dlim_{r\to+\infty} e^{2r}A'(r)=0$.
    Moreover, $J'$ is represented by
    \beaa
    J'(r) = \int^{+\infty}_{0} A'(r \cosh s) \cosh s ds. 
    \eeaa
\end{lemma}
\begin{proof}[Proof of Lemma \ref{lemm:Ad}]
{\red From the assumptions of Theorem \ref{thm:w11}, we see that $\dlim_{r\to+0} J'(r) + r J''(r) = 0$, and that}
there exists a constant $C_J$ such that
\begin{align*}
    &|J'(r) + rJ''(r)| \le C_J  \min\{r, r^{-\alpha} \} e^{-2r}
\end{align*}
for $r \in [0, \infty)$.
Using this boundedness, we obtain that 
\begin{align*}
    |A'(r)| 
    &\le \frac{2 C_J}{\pi}  \int_0^\infty (r \cosh s) e^{ - 2r \cosh s } ds \\
    & = \frac{2 C_J}{\pi} \Big( \int_r^{\sqrt{r^2 + 1}} t e^{-2t} \frac{ dt }{\sqrt{ t^2 - r^2 }}  + \int_{\sqrt{r^2 + 1}}^\infty t e^{-2t} \frac{ dt }{\sqrt{ t^2 - r^2 }}  \Big) \\
    & \le \frac{2 C_J}{\pi} \frac{5}{4}
\end{align*}
for $r \ge 0$.
For any $r_n \in [0,+\infty)$, we set $\mathcal{J}_n(s):=J'(r_n \cosh s) + r_n(\cosh s) J''(r_n \cosh s)$ for $s\in [0,+\infty)$.
Then, for any $ r_n \to r, (n \to +\infty)$, we have $\dlim_{n\to+\infty}\mathcal{J}_n(s)= J'(r \cosh s) + r(\cosh s) J''(r \cosh s)$ for a fixed $s$ from the continuity of $J'$ and $J''$.
Thus, using the dominated convergence theorem, we see that $A'$ is continuous for $r\ge 0$.
Moreover, for any $ r_n \to +0, (n \to \infty)$, we have 
\begin{align*}
\lim_{n\to+\infty} A'(r_n) 
= \lim_{n\to+\infty} - \frac{2}{\pi} \int_0^\infty \mathcal{J}_n(s) ds
=  - \frac{2}{\pi} \int_0^\infty\lim_{n\to+\infty} \mathcal{J}_n(s)  ds=0.    
\end{align*}
Similarly to the proof of Lemma \ref{lemm:A}, we can compute that 
\begin{equation*}
    |A'(r)| \le \frac{2C_J r }{\pi} \min\{ 1, r^{-1-\alpha} \} M_1(2 r).
\end{equation*}
This yields that $\dlim_{r\to +\infty }e^{2r}A'(r) =0$.
Finally, we see that
\begin{align*}
    \int^{+\infty}_{0} A'(r \cosh s) \cosh s ds
    = -\frac{2}{\pi r} \int_r^\infty \int_r^\eta \frac { s ( J'(\eta) + \eta J''(\eta) )  }{\sqrt{ s^2 - r^2 } \sqrt{ \eta^2 - s^2 } } ds d\eta 
    =J'(r).
\end{align*}
\end{proof}

From this lemma, we define $h'(0)=\dlim_{r\to+\infty} - e^{2r} ( A(r) + A'(r) ) $.
Then, we see that 
\begin{align*}
&h'(\lambda) = - \frac{ A'(-\log \lambda) + A(-\log  \lambda) }{\lambda^2} = -e^{2r}( A'(r) + A(r) ), \ (-\log \lambda = r).
\end{align*}

Now, we can estimate that 
\begin{align*}
& \sum_{j=1}^2 \left\| \frac{ \partial }{ \partial x_j } ( K - K_{N+1} ) \right\|_{L^1} 
= 8 \int^{+\infty}_{0}  r \left| J'(r) + \sum_{j=1}^{N+1} \frac{ j^3 \alpha_j }{2\pi} M_1(jr) \right| dr \\
&\le 8 \int^{+\infty}_{0} \int^{+\infty}_{0} r \cosh s e^{-2r \cosh s} \\
&\quad \quad \quad \times\left|  A'(r\cosh s)e^{ 2r \cosh s} +  \sum_{j=1}^{N+1} j \alpha_{j} c_{j-1,n} e^{-(j-2) r \cosh s} \right| ds dr \\
&\le 8 \Big( \int^{+\infty}_{0} \int^{+\infty}_{0} r \cosh s e^{-2r \cosh s}  ds dr \Big) \\
&\quad \quad \quad \times \max_{R\ge 0}   \left| e^{ 2 R }( A'(R)  + A( R ) )  +  \sum_{j=0}^{N} j \alpha_{j+1} c_{j,n} e^{-(j-1) R} \right| \\
& \quad + 8 \Big( \int^{+\infty}_{0} \int^{+\infty}_{0} r \cosh s e^{-r \cosh s}  ds dr \Big)  \max_{R\ge 0}   \left| e^{ R } A( R )  -  \sum_{j=0}^{N} \alpha_{j+1} c_{j,n} e^{-j R} \right| \\
& = \pi \max_{\lambda \in[0,1]} | h'(\lambda) - P'_N(\lambda) | + 4\pi   \max_{\lambda \in[0,1]} | h(\lambda) - P_N(\lambda) |, \quad (\lambda = e^{-R}).
\end{align*}
Thus, we obtain the desired assertion in two-dimensional case.
\end{proof}

Here, we introduce some examples of $A$.

\begin{example}
 Let $K(x)=J(|x|)=(a+b|x|)e^{-c|x|}$ with $a,b\in\bR$ and $c>1$. 
 Then, $K$ satisfies Assumption \ref{ass:kernel}, and $A$ is represented by
 \beaa
    A(r) &=& -\dfrac{2r}{\pi} \int^{+\infty}_{0}  \{ (b-ac) - bc r\cosh s \}e^{-cr \cosh s} ds \\
    &=& \dfrac{2r}{\pi} \{ (ac-b)M_{0}(cr) + bc r M_{1}(cr)\}.
 \eeaa
\end{example}


\begin{example}
  Let $K(x)=J(|x|)=e^{-a|x|^2}$ with $a>0$. 
  It is easy to see that $K$ satisfies Assumption \ref{ass:kernel}.
  Then, $A$ is computed as
 \beaa
    A(r) &=& \dfrac{ 4 ar^2}{\pi} \int^{+\infty}_{0}  e^{-ar^2 \cosh^2 s} \cosh s ds \\
    &=& \dfrac{ 2 ar^2}{\pi} \int^{+\infty}_{0}  \dfrac{e^{-ar^2 (s+1)}}{\sqrt{s}}  ds \\
    &=& \dfrac{2ar^2}{\pi} e^{-ar^2} \sqrt{\dfrac{\pi}{ar^2}} = 2\sqrt{\dfrac{a}{\pi}} r e^{-ar^2}.
 \eeaa
\end{example}

The following function with compact support does not satisfy Assumption \ref{ass:kernel}. 
However, since we use it in the numerical simulation as in Fig. \ref{fig:Hap}, we give the calculation of $A$.
\begin{example}\label{ex:Hp}
Let $K$ be the following function:
\[
K(x)=(1-|x|)\chi_{B(1)}(x)
=\begin{cases}
    1-|x|, \ &\mbox{if} \ x \in B(1), \\
    0, &\mbox{otherwise}.
\end{cases}
\]
\end{example}
Then, $A$ is provided by
\[
A(r) 
= \frac{2r}{\pi}\log\Big( \frac{1}{r} + \sqrt{ \frac{1}{r^2} - 1 } \Big) \chi_{B(1)}(r), \quad r = |x|.
\]

\subsubsection{Three-dimensional case}

Finally, we prove Theorem \ref{thm:est_poly} for the case that $n=3$.
It should be noted that in this case, $k_j$ is represented as 
\beaa
k_j (x) = \dfrac{j^2}{4\pi |x|} e^{-j|x|}.
\eeaa

\begin{proof}[Proof of Theorem \ref{thm:est_poly}]
Let Assumption \ref{ass:kernel} for $n=3$ be enforced.
Then, we obtain
\beaa
\left\| K - K_{N+1} \right\|_{L^1} &=& 4\pi \int^{+\infty}_{0} r^2 \left| J(r) - \dsum^{N+1}_{j=1} \dfrac{j^2 \alpha_j}{4\pi r} e^{-jr} \right|dr \\
&=& 4\pi \int^{+\infty}_{0} re^{-r} \left| 
r e^{r} J(r) - \dsum^{N}_{j=0} \alpha_{j+1} c_{j,n} e^{-jr} \right|dr \\
&\le& 4\pi \max_{\lambda\in[0,1]}\left|h(\lambda) - P_{N}(\lambda) \right|
\eeaa
from the definition of $\{\alpha_j\}_{1\le j\le N+1}$.
The desired assertion in three-dimensional case is obtained.
\end{proof}

By a quite similar argument, we have a point-wise absolute error. 
\begin{corollary}\label{cor:d3}
Let $n=3$ and $d_j = j^{-2}$ for $j\in\bN$. 
Let Assumption \ref{ass:kernel} be enforced.
Then, we have
\beaa
\left|K(x) - K_{N+1}(x) \right| \le  \dfrac{e^{-|x|}}{|x|} \max_{\lambda\in[0,1]} \left|h(\lambda) - P_{N}(\lambda) \right|
\eeaa
for all $x\in\bR^3\backslash\{0\}$.
\end{corollary}

\begin{proof}[Proof of Theorem \ref{thm:w11}]
{\red From the assumption that $\dlim_{r\to+0} ( J(r) + r J'(r))$ exists, we see that $\dlim_{r\to+0} r J(r)$ exists.}
We note that 
\begin{align*}
h'(\lambda) 
&= - \frac{ J(-\log \lambda) - (\log  \lambda) J'(-\log  \lambda)  - (\log  \lambda) J(-\log  \lambda) }{\lambda^2} \\
&= -e^{2r}( J(r) + rJ'(r) + rJ(r) ), \ (-\log \lambda = r).
\end{align*}
We can define $h'(0) = \dlim_{r \to +\infty} -e^{2r}( J(r) + rJ'(r) + rJ(r) )$ and $h'(1) = \dlim_{r \to +0} -e^{2r}( J(r) + rJ'(r) + rJ(r) )$.
Using Corollary \ref{cor:d3}, we obtain that 
\begin{align*}
	\left\| (K - K_{N+1})/|\cdot|  \right\|_{L^1} 
	&\le 4\pi \Big( \int^{+\infty}_{0} r e^{-r}  dr \Big) \max_{\lambda\in[0,1]} \left|h(\lambda) - P_{N}(\lambda) \right|\\
	&= 4\pi   \max_{\lambda\in[0,1]} \left|h(\lambda) - P_{N}(\lambda) \right|.
\end{align*}
Additionally, using the notation $J_N(r) :=K_N(|x|) = \sum_{j=1}^N a_j k_j(|x|), \ r=|x|$, we see that 
\begin{align*}
J'_{N+1}( r ) =  -  \sum_{j=0} ^{N} j \alpha_{j+1} k_{j+1}(r) -J_{N+1}(r) - \frac{1}{ r } J_{N+1} (r).
\end{align*}
Then, we can estimate that
\begin{align*}
&\sum_{j=1}^3   \left\| \frac{ \partial }{ \partial x_j } ( K - K_{N+1} ) \right\|_{L^1} 
=  6 \pi  \int^{+\infty}_{0}  r^2  \left|   J'(r)  - J'_{N+1}(r)  \right| dr \\
& \le  6 \pi  \int^{+\infty}_{0}  r^2  \left|  J(r) + J'(r) + \frac{1}{ r } J(r) - J_{N+1}(r) - J'_{N+1}(r) - \frac 1 { r } J_{N+1} (r) \right| dr \\
& \quad + \frac{3 } 2   \left\|  K - K_{N+1}  \right\|_{L^1}  + \frac{3 } 2   \left\| ( K - K_{N+1} )/| \cdot | \right\|_{L^1}  \\
&\le 6\pi \Big( \int^{+\infty}_{0}  r e^{-2r }  dr \Big) \max_{ \lambda \in [0,1] } | h'(\lambda) - P_N'(\lambda) |  + 12 \pi \max_{ \lambda \in [0,1] } | h(\lambda) - P_N(\lambda) | \\
&= \frac{3\pi} 2  \max_{ \lambda \in [0,1] } | h'(\lambda) - P_N'(\lambda) | + 12 \pi \max_{ \lambda \in [0,1] } | h(\lambda) - P_N(\lambda) |. 
\end{align*}
Thus, we obtain the desired assertion in three-dimensional case.
\end{proof}

\subsection{Bernstein Polynomial}\label{sec:BP}
We discuss the approximation of integral kernels using polynomial approximations.
Here, we review a polynomial approximation by the Bernstein polynomial.

Let $h:[0,1]\to\bR$ and $N\in\bN$.
We denote the Bernstein polynomial of degree $N$ to the function $h$ by
\beaa
B_{N}[h](\lambda) := \dsum^{N}_{\nu=0} h\left( \dfrac{\nu}{N} \right) b_{\nu,N}(\lambda),\
 b_{\nu,N}(\lambda):= \dbinom{N}{\nu} \lambda^{\nu} (1-\lambda)^{N-\nu},\ (\nu=0,1,\ldots,N).
\eeaa
We now proceed to polynomial expansion of $B_{N}[h]$.
Since we find
\beaa
b_{\nu,N}(\lambda) &=& \dbinom{N}{\nu} \dsum^{N-\nu}_{j=0} (-1)^{j}\dbinom{N-\nu}{j}  \lambda^{j+\nu}  
= \dbinom{N}{\nu} \dsum^{N}_{j=\nu} (-1)^{j-\nu}\dbinom{N-\nu}{j-\nu}  \lambda^{j}  \\
&=&\dsum^{N}_{j=\nu} (-1)^{j-\nu}\dbinom{N}{j}\dbinom{j}{\nu}  \lambda^{j},\quad (\nu=0,1,\ldots,N),
\eeaa
we get
\beaa
B_{N}[h](\lambda) &=& \dsum^{N}_{\nu=0} \dsum^{N}_{j=\nu} (-1)^{j-\nu} h\left( \dfrac{\nu}{N} \right)  \dbinom{N}{j}\dbinom{j}{\nu}
 \lambda^{j} \\
&=& \dsum^{N}_{j=0} \left( \dsum^{j}_{\nu=0} (-1)^{j-\nu}  h\left( \dfrac{\nu}{N} \right)  \dbinom{N}{j}\dbinom{j}{\nu}  \right) \lambda^{j} = \dsum^{N}_{j=0} \beta_{j,N}[h] \lambda^{j},
\eeaa
where $\{\beta_{j,N}[h]\}_{0\le j\le N}$ is defined by \eqref{coef}.
The following theorem of polynomial approximation is known.

\begin{lemma}[\cite{Lorentz}]\label{lem:poly}
For $m\in\{0,1\}$, there exists a constant $E(m)>0$ such that if $h\in C^{m}([0,1])$, then we have
\beaa
\max_{\lambda\in[0,1]} |h(\lambda)-B_N[h](\lambda)| \le E(m)N^{-m/2} \omega(h^{(m)},N^{-1/2}).
\eeaa
\end{lemma}
\begin{remark}
    It is known that we can choose $E(0)=5/4$ and $E(1)=3/4$ from \cite{Lorentz}.
    Moreover, it is also known that the convergence order becomes faster with the smoothness of $h$.
\end{remark}


Convergence of derivatives has also been reported by \cite{Floater,Lorentz}.
Especially, the following result is known with respect to the order of convergence:
\begin{lemma}[\cite{Floater}]
If $h\in C^{m+2}([0,1])$ for some $m\ge 0$, then 
\beaa
&&\max_{\lambda\in[0,1]}\left| \dfrac{d^m}{d\lambda^m}\big[h(\lambda)- B_N[h](\lambda) \big] \right|\\
&&\le 
\dfrac{1}{2N} \left( m(m+1) \max_{\lambda\in[0,1]}|h^{(m)}(\lambda)|+  m  \max_{\lambda\in[0,1]}|h^{(m+1)}(\lambda)|+ \dfrac{1}{4}\max_{\lambda\in[0,1]}|h^{(m+2)}(\lambda)| \right).
\eeaa
\end{lemma}

\subsection{Lagrange Polynomial}\label{sec:Lag}
In this subsection we introduce another candidate for the polynomial $P_N$ by using the Lagrange interpolation polynomial with the Chebyshev nodes.
We utilize the result of the coefficient determination of the Lagrange polynomial in the case on $[0,1]$ by \cite{MT}.
We firstly prepare the notations.
We set 
\begin{align*}
    C_{k,N} :=  (-1)^k  2^{N-2k-1} \frac{N}{N-k} \binom{N-k}{k}, \ (k=0,1, \ldots, \Big[\frac N 2 \Big] ), \ N \in \bN,
\end{align*}
where $[\cdot]$ is the Gauss symbol.
Using this notation to the Chebyshev polynomial, we obtain the expression $ T_N(x) = \sum_{k=0}^{[N/2]}C_{k,N} x^{N-2k}$.
Next, we prepare the following constants: 
\begin{equation*}
    \mu^{(N)}_{k,j}:=  (-1)^j 2 ^{ N -2k-j}\begin{pmatrix} N - 2k \\ j \end{pmatrix} C_{k,N}, \ N \in \bN,
\end{equation*}
\begin{equation*}
    \xi_{k,N} :=
    \left\{
    \begin{aligned}
        &\sum_{ \nu=0}^{[N/2]-[(k+1)/2]}  \mu^{(N)}_{\nu,N-2\nu-k}, \quad \text{if $N$ is even}, \\ 
        &\sum_{ \nu=0}^{[N/2]- [ k/2 ] }  \mu^{(N)}_{\nu,N-2\nu-k}, \quad \text{otherwise}. \\ 
    \end{aligned}
    \right.
\end{equation*}
Utilizing these coefficients, with respect to the shifted Chebyshev polynomial we have $T_N( 2x - 1 ) = \sum_{k=0}^N \xi_{k,N} x^k, \ x \in [0,1].$
This proof is written in \cite{MT}.

Let us denote the Chebyshev nodes by
\begin{equation*}
    r_{j,N}:= \frac{1}{2} + \frac{1}{2}\cos \frac{2j+1}{2N}\pi, \quad (j=0,1,\ldots, N-1).
\end{equation*}
For the function $h$ defined in Assumption \ref{ass:kernel}, setting 
\begin{equation*}
    \zeta_{j,N} [h]:= \frac{h( r_{j,N+1} )}{\prod_{k=0,k\neq j}^N (r_{j,N+1} - r_{k,N+1} )}, \quad (j=0,1,\ldots, N), \ N \in \bN,
\end{equation*}
we define the Lagrange polynomial as
\begin{equation*}
    L_N(\lambda):= \sum_{j=0}^N  \zeta_{j,N} \prod_{k=0,k\neq j}^N (\lambda - r_{k,N+1} ).
\end{equation*}
Regarding the determination of the coefficients of the Lagrange polynomial given interpolation points using Chebyshev nodes, the following lemma holds.
\begin{lemma}[\cite{MT}]
For $N \in \bN$, we set  the coefficients as
    \begin{align*}
        &\tau_{j,\nu}^{(N)} := \sum_{k=j}^{N-1} (r_{\nu,N})^{k-j} \xi_{k+1,N}, \quad (j=0,1,\ldots,N-1), \ (\nu=0,1,\ldots,N-1),\notag\\
        &l_{j,N} [h]:= \frac{1}{2^{2N+1}}  \sum_{ \nu=0}^N \zeta_{\nu,N} [h] \tau_{j,\nu}^{(N+1)}, \quad (j=0,1,\ldots,N). 
    \end{align*}
Then, it holds
    \begin{equation*}
        L_N(\lambda)= \sum_{j=0}^N l_{j,N} [h] \lambda^j, \quad \lambda \in [0,1].
    \end{equation*}
\end{lemma}
Using the properties of the Lagrange and Chebyshev polynomial as in \S 6.5 in \cite{MH2003} and \S 1.3 in \cite{R1990}, we obtain the following estimates.
\begin{lemma}\label{lemm:Lag}
For $h \in \mathrm{Lip}([0,1])$, it holds that
\begin{align*}
    &\max_{ \lambda \in [0,1] } |  h(\lambda) - L_N(\lambda) |
   \le   (1 + \mu_N )  \omega ( h, N^{-1} ),
\end{align*}
where $\mu_N = (2/\pi)\log N + 1$ that comes from the Lebesgue constant.
\end{lemma}
Furthermore, for the smooth functions, the following estimate is known.
\begin{lemma}[\cite{C1998}, \S 4.6]\label{lem:TLag_sm}
For $h \in C^{m}([0,1])$ with $m \in \bN$, it holds that
\begin{equation*}
    \max_{ \lambda \in [0,1] } | h(\lambda) - L_N(\lambda) |
    \le \frac{1}{2} \Big( \frac{\pi}{2}\Big)^m \|h^{(m)} \|_{C([0,1])} (N-m+2)^{-m}, \quad N\ge m.
\end{equation*}
\end{lemma}

\subsection{Polynomial approximation and numerical examples}
We first provide the proof of Corollary \ref{cor:ker}. 

\begin{proof}[Proof of Corollary \ref{cor:ker}]
Setting $\alpha_j = \beta_{j-1,N}[h]/c_{j-1,n}$ for $j=1,2,\ldots,N$ and using  Theorem \ref{thm:est_poly} and Lemma \ref{lem:poly}, we obtain the first assertion.
Similarly, setting $\alpha_j = l_{j-1,N}[h]/c_{j-1,n}$ for $j=1,2,\ldots,N$ and using  Theorem \ref{thm:est_poly} and Lemma \ref{lemm:Lag}, we obtain the second assertion.
\end{proof}

Now, we present numerical examples of the 
approximation of a kernel.
We first treat the case that the kernel is given by
\beaa
K(x)= e^{-|x|^2}.
\eeaa
Note that $K$ satisfies Assumption \ref{ass:kernel} for all $n \in \{ 1,2,3\}$.
The graphs of $K$ and $K_{N+1}$ using the Bernstein polynomial are shown in Fig. \ref{fig:app}.
It can be visually confirmed that the graphs of $K$ and $K_{N+1}$ are similar in the case that $N=30$.
When $n=3$, the graphs are a bit separated near the origin,
which is expected due to the singularity of the Green function there.
The values in the middle of $\{\alpha_j\}$ are large, on the order of $10^{4}$ or more.
Since $\{\alpha_j\}$ are defined by using \eqref{coef},
increasing $j$ is thought to cause it by increasing the value of the binomial coefficient.

Secondly, Fig. \ref{fig:Lag} shows the numerical results by using the Lagrange polynomial with Chebyshev nodes.
It is observed that the graphs of $K$ and $K_{N+1}$ are similar in the case that $N=10$.

\begin{figure}[bt] 
\begin{center}
	\includegraphics[width=12cm, bb= 0 0 613 221]{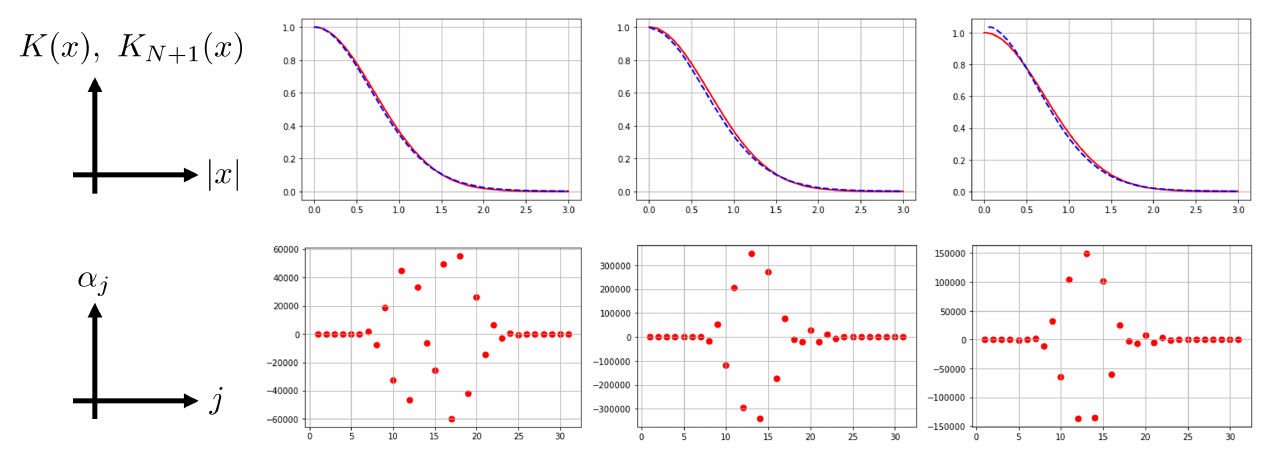}
\end{center}
\caption{\small{
Graphs of $K(x)=e^{-|x|^2}$ and $K_{N+1}$ with $N=30$ (top panels), the distributions of $\{\alpha_j\}_{1\le j\le N+1}$ (bottom panels) by using the Bernstein polynomial.
The graph of $K$ and $K_{N+1}$ are shown by the solid line and the dashed line, respectively.
The numerical examples correspond to the results with $n=1$ (left), $n=2$ (middle), and $n=3$ (right), respectively.
}}
\label{fig:app}
\end{figure}
\begin{figure}[bt] 
\begin{center}
	\includegraphics[width=11cm, bb= 0 0 1174 446]{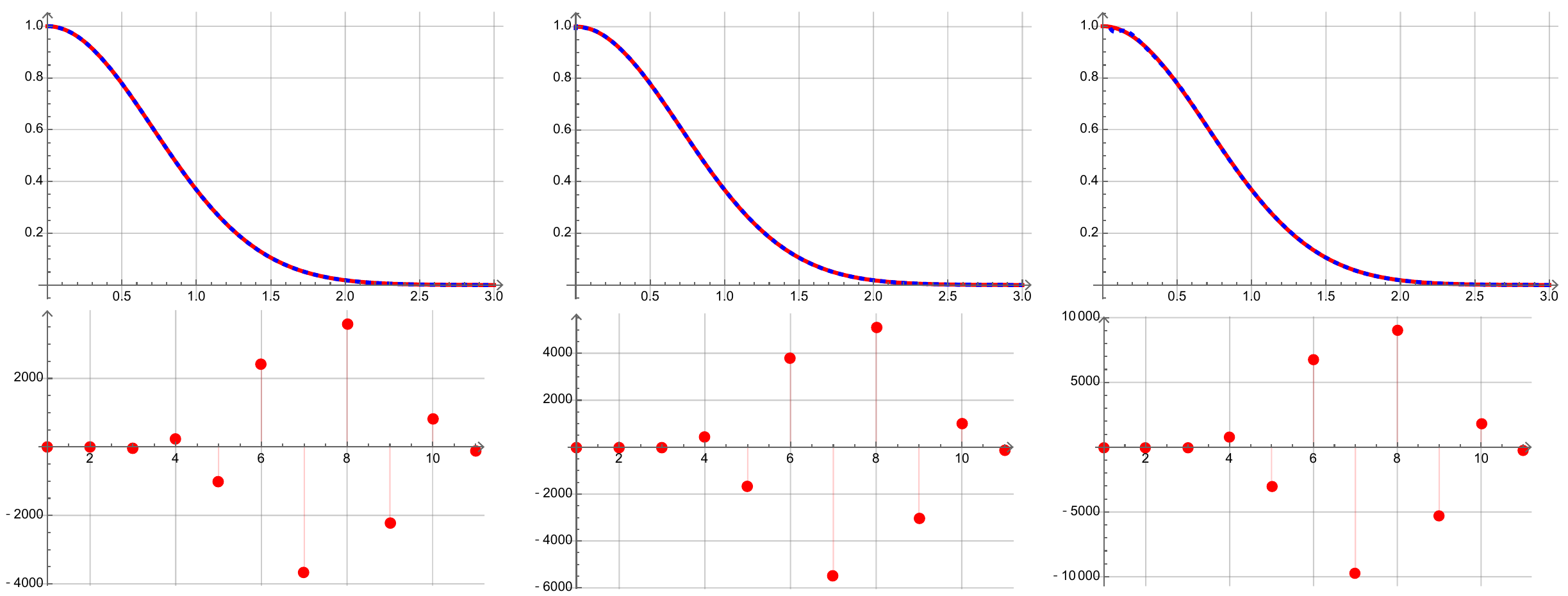}
\end{center}
\caption{\small{Graphs of $K=e^{-|x|^2}$ and $K_{N+1}$ with $N=10$ (top panels), the distributions of $\{\alpha_j\}_{1\le j\le N+1}$ (bottom panels) by using the Lagrange polynomial with the Chebyshev nodes.
The graph of $K$ and $K_{N+1}$ are shown by the solid line and the dashed line, respectively.
The numerical examples correspond to the results with $n=1$ (left), $n=2$ (middle), and $n=3$ (right), respectively.
}}
\label{fig:Lag}
\end{figure}
\begin{figure}[bt] 
\begin{center}
	\includegraphics[width=11cm, bb= 0 0 1174 434]{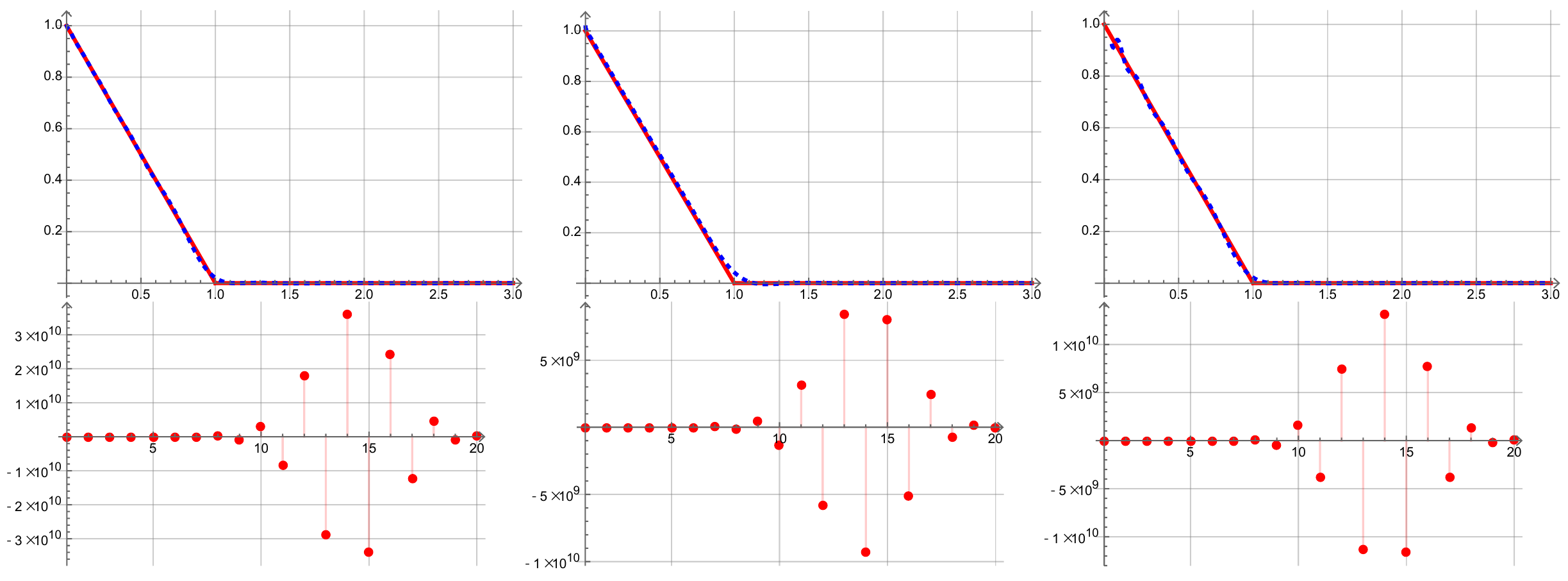}
\end{center}
\caption{\small{Graphs of $K(x)=(1-|x|)\chi_{B(1)}(x)$ and $K_{N+1}$ with $N=19$ (top panels), the distributions of $\{\alpha_j\}_{1\le j\le N+1}$ (bottom panels) by using the Lagrange polynomial with the Chebyshev nodes.
The graph of $K$ and $K_{N+1}$ are shown by the solid line and the dashed line, respectively.
The numerical examples correspond to the results with $n=1$ (left), $n=2$ (middle), and $n=3$ (right), respectively.
}}
\label{fig:Hap}
\end{figure}

Figure \ref{fig:Hap} shows the numerical results for the case of an integral kernel with compact support as in Example \ref{ex:Hp}.
Only two-dimensional case does not satisfy Assumption \ref{ass:kernel}.
It can be observed that the approximation by the linear sum of the Green functions $K_{N+1}$ converges for the continuous function in each dimension.
This indicates that Assumption \ref{ass:kernel} in the case that $n=2$ may be relaxed.

\section{Reaction-diffusion approximation of nonlocal interactions}\label{sec:main}
 Now we explain the proof of the main results.
 
\begin{proof}[Proof of Theorem \ref{thm:rda}]
{\red Let $K\in L^1(\bR^n)$ be a radial function.
From Theorem~\ref{thm:ker-gen},
for any $\eps > 0$, there exist $N \in \bN$, constants $\{ \alpha_j \}_{1\le j\le N}$ and positive constants $\{ d_j \}_{1\le j\le N}$ such that $\| K -K_N \|_{L^1} \le \eps$.
}
Then, from Lemma \ref{lem:err1} there exists $C_{3}= C_{3}(f,K,T)>0$ such that
\beaa
\|u(t) - u^{0}(t) \|_{L^p} \le C_{3}\|u_0\|_{L^p} \|K-K_N\|_{L^1} \le C_{3}\|u_0\|_{L^{p}} \eps
\eeaa
holds for any $t\in (0,T]$.
Moreover, there exists a positive constant \\
$C_{1}= C_{1}(f,D,\{\alpha_j\}_{1\le j\le N},\{d_j\}_{1\le j\le N},T)$ such that
\beaa
    \|u^{\delta}(t)-u^{0}(t)\|_{L^p} &\le& C_{1}\|u_0\|_{L^p}\delta, \\
    \|v^{\delta}_{j}(t)-(k_j*u^{0})(t)\|_{L^p} &\le& C_{1}\|u_0\|_{L^{p}} \delta    
    \eeaa
hold for any $t\in (0,T]$ from Proposition \ref{lem:err2}.
Thus, for any $t\in (0,T]$, we have
\beaa
\|u(t)-u^{\delta}(t)\|_{L^p} &\le & \|u^{\delta}(t)-u^{0}(t)\|_{L^p}+ \|u(t) - u^{0}(t) \|_{L^p}  \\
&\le & (C_{1}\delta+C_{3}\eps)\|u_0\|_{L^{p}}
\eeaa
and
{\red
\beaa
\|v^{\delta}_j(t) -(k_j* u)(t)  \|_{L^p} &\le& \|v^{\delta}_j(t) -(k_j* u^0)(t)  \|_{L^p} +  \|(k_j*u^0)(t) -(k_j* u)(t)  \|_{L^p} \\
&\le & (C_{1}\delta+C_{3}\eps)\|u_0\|_{L^{p}}
\eeaa
}
from the Young inequality.
Therefore, the proof is complete.
\end{proof}

\section{Concluding remarks}\label{sec:con}
{\red In this paper, we have demonstrated that in any Euclidean space, the solution of the nonlocal evolution equation with any radial integral kernel can be approximated by that of a reaction-diffusion system with auxiliary factors.}
We showed that the reaction-diffusion system coupled with auxiliary activators and inhibitors can approximate the time evolution governed by  arbitrary nonlocal interactions over any finite time interval.
This is achieved by considering the quasi-steady state of the auxiliary factors.
{\red Within this framework, the parameters of the reaction–diffusion system can be explicitly specified according to the shape of the integral kernel, at least in spatial dimensions up to three.}

Cell-cell interactions can be formulated as an integral kernel with compact support as in Fig. \ref{fig1} (a).
Cell can realize the interaction necessary to generate the pattern by changing the way they make contact.
The results in this paper indicates that this change may be equivalent to precisely tuning parameters in multi-component reaction-diffusion system in a high-dimensional space as in Fig. \ref{fig:Hap}.
In this sense, cell-cell interactions are a reasonable mechanism for creating patterns.
On the other hand, reaction-diffusion systems are versatile because they can reproduce phenomena with seemingly different mechanisms in high-dimensional spaces. 
Therefore, in high-dimensional space, reaction, diffusion and quasi-steady state can produce a variety of nonlocal interactions.

Based on our methodology, nonlocal problems can be reformulated within the framework of reaction-diffusion systems, and {\it vice versa}.
For instance, the theoretical framework of the $n$-component reaction-diffusion system can be applied to analyze the nonlocal problems.
Conversely, insights from nonlocal problems can also be leveraged to study multi-component reaction-diffusion systems.


The key aspect in this reaction-diffusion approximation is the approximation result to any radial kernel by a linear sum of the Green function $k_j$ in $L^1(\R^n)$ space.
{\red The error term can, in general, be represented by a polynomial approximation, valid up to three spatial dimensions.}
We employed the Bernstein and Lagrange polynomials to obtain the convergence.
The other polynomials that have good properties for the approximation can be candidates for $L^1$ convergence.
Regarding this convergence result, we remark that the result of the expansion by $k_j$ in the case of $H^{m}(\mathbb{R}^n)$ space with arbitrary $n \in \bN$ and $m \in \bN \cup \{ 0 \}$ is obtained by \cite{IT}.
As in Remark \ref{rem:Cm}, if $h\in C^m([0,1])$ for $m \in \bN$, the error $\eps$ converges to $0$ in the order $O(N^{-m})$.
Thus, it is possible to obtain numerical solutions at low cost and quickly, especially for nonlocal evolution equations with smooth integral kernels by applying this approximation method to the numerical simulations.

For the equations with advective nonlocal interactions such as cell adhesion model \eqref{eq:cell-ad} and the nonlocal Fokker-Planck equation, the expansion result by the Green function in $W^{1,1}(\R^n)$ is necessary for this type of approximation by system of partial differential equations.
Therefore, Theorem \ref{thm:w11} can also be useful in the context of such study.

\section*{Acknowledgments}
The authors were partially supported by JSPS KAKENHI Grant Number 24H00188.
HI was partially supported by JSPS KAKENHI Grant Numbers 23K13013 {\red and 25K07142}. YT was partially supported by JSPS KAKENHI Grant Number 22K03444 and 24K06848.

\section*{Conflict of interest}
The authors declare that they have no conflict of interest.

\begin{appendices}
\renewcommand{\thesection}{\Alph{section}} 
\makeatletter
\renewcommand\@seccntformat[1]{\appendixname\ \csname the#1\endcsname.\hspace{0.5em}}
\makeatother

\section{Existence of a global solution to nonlocal problem}
\label{app:non}

We first consider the existence of the mild solution to \eqref{pro:non}.
The case that $p=+\infty$ is simpler than the case that $1\le p<+\infty$, thus the discussion is omitted here.
For $T>0$ and $1\le p< +\infty$, we define the Banach space $X(T,p):= C([0,T];BC(\bR^n)\cap L^p(\bR^n))$ and the norm
\beaa
\|\phi\|_{X(T,p)}:= \sup_{0\le t\le T}\|\phi(t)\|_{L^{\infty}} + \sup_{0\le t\le T}\|\phi(t)\|_{L^{p}}.
\eeaa
We first check the property of $\cP[\phi]$.

\begin{lemma}
    $\cP:X(T,p)\to X(T,p)$ is well-defined. Moreover, for any $u_0\in BC(\bR^n)\cap L^p(\bR^n)$, we have
    \beaa
    \dlim_{t\to +0} \|\cP[\phi](t) - u_0\|_{L^{p}} =0
    \eeaa
    and 
    \beaa
    \dlim_{t\to +0}\cP[\phi](t,x) = u_0(x)
    \eeaa
    for any $x\in\bR^n$.
\end{lemma}
\begin{proof}
    Let $\phi\in X(T,p)$ and $t\in (0,T]$.
    Then, we obtain that
    \beaa
    \|\cP[\phi](t)\|_{L^{p}} &\le & \| u_0 \|_{L^p} + C_f (1+\|K\|_{L^1}) \int^{t}_{0}\|\phi(s)\|_{L^p}ds \\
    &\le & \| u_0 \|_{L^p} + C_f T (1+\|K\|_{L^1}) \|\phi\|_{X(T,p)}
    \eeaa
    from the Young inequality.
    A similar argument yields that
    \beaa
    \|\cP[\phi](t)\|_{L^{\infty}} &\le & \| u_0 \|_{L^\infty} + C_f T (1+\|K\|_{L^1}) \|\phi\|_{X(T,p)}.
    \eeaa
    Moreover, since $\phi(t)$, $(K*\phi)(t)$ and $H(t,D)*u_0$ belong to $BC(\bR^n)$, we have $\cP[\phi](t)\in BC(\bR^n)$. 
    Thus, the desired assertion is verified if the convergence of the initial datum is obtained.
    Since $u_0$ is a bounded continuous function in $L^p(\bR^n)$,
    it is known that
    \beaa
    \dlim_{t\to +0} \|H(t;D)*u_0 - u_0\|_{L^{p}} =0
    \eeaa
    and 
    \beaa
    \dlim_{t\to+0}(H(t;D)*u_0)(x) = u_0(x)
    \eeaa
    hold for any $x\in\bR^n$. Therefore, we obtain
    \beaa
    \|\cP[\phi](t) - u_0\|_{L^{p}} 
    & \le& \|H(t;D)*u_0 - u_0\|_{L^{p}} + C_f (1+\|K\|_{L^1}) \int^{t}_{0}\|\phi(s)\|_{L^p} ds \\
    &\le&  \|H(t;D)*u_0 - u_0\|_{L^{p}}+ C_f t (1+\|K\|_{L^1}) \|\phi\|_{X(T,p)} \\
    &\to& 0 \quad (t\to+0).
    \eeaa
    Similarly, for any $x\in\bR^n$, we have
    \beaa
    |\cP[\phi](t,x) - u_0(x)| 
    &\le& |(H(t;D)*u_0)(x)-u_0(x)| +C_f t (1+\|K\|_{L^1}) \|\phi\|_{X(T,p)} \\
    &\to& 0 \quad (t\to+0). 
    \eeaa
    The proof is complete.
\end{proof}

Next, we show the existence of the mild solution to \eqref{pro:non}.

\begin{proposition}
    There exists a unique mild solution $u\in C([0,T];BC(\bR^n)\cap L^p(\bR^n))$ to \eqref{pro:non} with an initial datum
    $u_0\in BC(\bR^n)\cap L^p(\bR^n)$. 
\end{proposition}
\begin{proof}
    For $\rho>0$, we introduce the norm 
    \beaa
    \|\phi\|_{\rho}:= \sup_{0\le t\le T}e^{-\rho t}\|\phi(t)\|_{L^{\infty}} + \sup_{0\le t\le T}e^{-\rho t} \|\phi(t)\|_{L^{p}}.
    \eeaa
    We note that it is an equivalent norm to $\|\cdot\|_{X(T,p)}$.

    Fix $\rho> C_f (1+ \|K\|_{L^1})$ and $t\in (0,T]$.
    For $\phi_1,\phi_2 \in X(T,p)$, we have
    \beaa
    &&|\cP[\phi_1](t)- \cP[\phi_2](t)| \\
    &&\quad\quad \le  C_f \int^{t}_{0} H(t-s;D)* \left( |\phi_1(s)-\phi_2(s)| + | (K*\phi_1)(s) -(K*\phi_2)(s))| \right) ds.
    \eeaa
    Thus, we deduce
    \beaa
    &&e^{-\rho t}\|\cP[\phi_1](t)-\cP[\phi_2](t)\|_{L^p} \\
    &&\quad\quad \le  C_f \int^{t}_{0}e^{-\rho t} \left( \|\phi_1(s)-\phi_2(s)\|_{L^p} + \| (K*\phi_1)(s) -(K*\phi_2)(s)) \|_{L^p} \right) ds \\
    &&\quad\quad \le  C_f (1+ \|K\|_{L^1}) \int^{t}_{0} e^{-\rho (t-s)} e^{-\rho s} \|\phi_1(s)-\phi_2(s)\|_{L^p} ds\\
    &&\quad\quad
    = \dfrac{ C_f (1+ \|K\|_{L^1})}{\rho} \sup_{0\le t\le T} e^{-\rho t} \|\phi_1(t)-\phi_2(t)\|_{L^p}
    \eeaa
    from the Young inequality.
    A similar argument yields 
    \beaa
    \|\cP[\phi_1](t)-\cP[\phi_2](t)\|_{L^\infty} \le \dfrac{ C_f (1+ \|K\|_{L^1})}{\rho} \sup_{0\le t\le T}\|\phi_1(t)-\phi_2(t)\|_{L^\infty}.
    \eeaa
    Therefore, we conclude
    \beaa
    \|\cP[\phi_1]-\cP[\phi_2]\|_{\rho} \le \dfrac{ C_f (1+ \|K\|_{L^1})}{\rho} \|\phi_1-\phi_2\|_{\rho}.
    \eeaa
    Since $\cP:(X(T,p),\|\cdot \|_{\rho}) \to (X(T,p),\|\cdot \|_{\rho})$ is a contraction map, there is a unique fixed point $u\in X(T,p)$ from the Banach fixed point theorem.
\end{proof}

Fix $1\le p\le +\infty$ and $T>0$.
Let $u\in C([0,T];BC(\bR^n)\cap L^p(\bR^n))$ be the mild solution to \eqref{pro:non} with an initial datum $u_0\in BC(\bR^n)\cap L^p(\bR^n)$.
Next, we estimate the bound of the $L^q$ norm of the solution.
\begin{lemma}
For any $q\in [p,+\infty]$ and $T>0$, 
\beaa
\| u(t) \|_{L^q}  \le e^{C_f (1+ \|K\|_{L^1} ) t} \|u_0\|_{L^q}
\eeaa
holds for any $t \in [0,T]$.
\end{lemma}
\begin{proof}
    Let $q \in[p,+\infty]$ and $t\in [0,T]$. Then, we have
    \beaa
    \|u(t)\|_{L^{q}} &\le& \|u_0\|_{L^q} + \int^{t}_{0} \| f(u(s),(K*u)(s))\|_{L^q} ds \\
    &\le& \|u_0\|_{L^q} + C_f (1+ \|K\|_{L^1}) \int^{t}_{0} \|u(s)\|_{L^q} ds
    \eeaa
    from the Young inequality.
    The Gronwall inequality yields
    \beaa
    \|u(t)\|_{L^q} \le  e^{C_f (1+ \|K\|_{L^1}) t}\|u_0\|_{L^q}.
    \eeaa
    Thus, we obtain the desired assertion.
\end{proof}

Let us consider the regularity of the mild solution.
From the general theory to the heat equation, we know that
\beaa
    H(\cdot;D)*u_0 \in C^{\infty}((0,+\infty)\times\bR^n).
\eeaa
Thus, we consider the regularity of the Duhamel term
\beaa
I(t,x) &:=& \int^{t}_{0} H(t-s;D)* f(u(s),(K*u)(s)) ds \\
&=& \int^{t}_{0} \int_{\bR^n} H(t-s,x-y;D)* f(u(s,y),(K*u)(s,y)) dy ds. \\
\eeaa
For sake of simplicity, we set
\beaa
g(t,x) := f(u(t,x),(K*u)(t,x)).
\eeaa
It is easy to see that $g\in C([0,T];BC(\bR^n))$.

\begin{lemma}\label{lem:duhamel}
    It holds that $I(t) \in BC^{1}(\bR^n)$ for all $t\in (0,T]$. 
    Moreover, $I(\cdot,x)\in C^{\alpha}([\tau,T])$ holds for any $\alpha\in(0,1)$, $\tau\in(0,T)$ and $x\in\bR^n$.
\end{lemma}
\begin{proof}
    We first consider the spatial derivative. 
    Fix $j=1,2,\ldots,n$ arbitrary.
    Since there is a constant $C_{H1}>0$ independent of $j$ such that
    \beaa
    \left\| \dfrac{\partial H}{\partial x_j}(t;D) \right\|_{L^1} \le \dfrac{C_{H1}}{\sqrt{t}}
    \eeaa
    holds for any $t>0$,
    we have
    \beaa
    \left| \dfrac{\partial I}{\partial x_j}(t,x) \right|
    &\le & \int^{t}_{0} \int_{\bR^n} \left| \dfrac{\partial H}{\partial x_j}(t-s,x-y;D)*g(s,y) \right|dy ds \\
    &\le& \left(  \int^{t}_{0} \left\| \dfrac{\partial H}{\partial x_j}(t-s;D) \right\|_{L^1} ds  \right)  \sup_{0\le t\le T}\|g(t)\|_{L^\infty} \\
    &=& 2 C_{H1} \sqrt{t}  \sup_{0\le t\le T}\|g(t)\|_{L^\infty}
    \eeaa
    for any $t>0$ and $x\in\bR^n$ from differentiating under the integral sign and the Young inequality. 
    Thus, for all $t\in(0,T]$, we obtain $\dfrac{\partial I}{\partial x_j}(t)\in BC(\bR^n)$ and thus conclude $I(t) \in BC^1(\bR^n)$.

    Next, we fix $\alpha\in(0,1)$, $\tau\in(0,T)$ and $x\in\bR^n$.
    Let $t_1,t_2\in [\tau,T]$ with $t_2>t_1$.
    Since we know that there is a constant $C_{H2}>0$ such that
    \beaa
    \|\Delta H(t;D)\|_{L^1} \le \dfrac{C_{H2}}{t}
    \eeaa
    holds for any $t>0$, we find
    \beaa
    |I(t_2,x)-I(t_1,x)| 
    &\le& \int^{t_1}_{0} |H(t_2-s;D) - H(t_1-s;D)|* |g(s)| ds\\
    &&\quad + \int^{t_2}_{t_1} H(t_2-s;D)* |g(s)| ds  \\
    &\le& \left(\int^{t_1}_{0} \int^{t_2}_{t_1} \left\| \dfrac{\partial H}{\partial t}(\eta-s;D)\right\|_{L^1}  d\eta ds \right)  \sup_{0\le t\le T} \|g(t)\|_{L^\infty}\\
    &&\quad + (t_2-t_1) \sup_{0\le t\le T}\|g(t)\|_{L^\infty} \\
    &=& \left(\int^{t_1}_{0} \int^{t_2}_{t_1} \left\| \Delta H(\eta-s;D)\right\|_{L^1}  d\eta ds \right)  \sup_{0\le t\le T} \|g(t)\|_{L^\infty}\\
    &&\quad + (t_2-t_1) \sup_{0\le t\le T}\|g(t)\|_{L^\infty} \\
    &\le& C_{H2} \left( t_2\log t_2 - t_1 \log t_1 -(t_2-t_1)\log(t_2-t_1)\right)  \sup_{0\le t\le T} \|g(t)\|_{L^\infty}\\
    &&\quad + (t_2-t_1) \sup_{0\le t\le T}\|g(t)\|_{L^\infty}.
    \eeaa
    Thus, we conclude $I(\cdot,x)\in C^{\alpha}([\tau,T])$. 
\end{proof}

In addition, the following lemma is provided for H{\"o}lder estimate for the solution.
\begin{lemma}\label{lem:Holder}
    Let $\gamma\in (0,1)$.
    Then, there exists a positive constant $C=C(\gamma)$ such that
    \beaa
    \left[ \dfrac{\partial H(t)}{\partial x_j}*\phi \right]_{\gamma} \le C t^{-(1+\gamma)/2} \|\phi\|_{L^{\infty}}
    \eeaa
    for any $t>0$, $j=1,2,\ldots,n$ and $\phi\in L^{\infty}(\bR^n)$, where $[\phi]_{\gamma}$ is H{\"o}lder seminorm.
\end{lemma}
\begin{proof}
    Let $\phi\in L^{\infty}(\bR^n)$.
    Fix $t>0$ and $j=1,2,\ldots,n$ arbitrary.
    For all $x,z\in\bR^n$, we find
    \beaa
    &&\left| \left(\dfrac{\partial H}{\partial x_j}*\phi\right)(t,x) - \left(\dfrac{\partial H}{\partial x_j}*\phi\right)(t,z) \right| \\
    && \le \|\phi\|_{L^{\infty}} \int_{\bR^n} \left|\dfrac{\partial H}{\partial x_j}(t,x-y) - \dfrac{\partial H}{\partial x_j}(t,z-y) \right| dy \\
    &&= t^{-1/2} \|\phi\|_{L^{\infty}} \int_{\bR^n} \left|\dfrac{\partial H}{\partial x_j}(1,t^{-1/2} (x-z)-y) - \dfrac{\partial H}{\partial x_j}(1,y) \right| dy \\
    &&=: t^{-1/2} \|\phi\|_{L^{\infty}} S(t^{-1/2}(x-z)).
    \eeaa
    Since $S(x)$ is a non-negative bounded continuous function on $\bR^n$ and satisfies 
    \beaa
    S(x) \le |x| \int_{\bR^n}  \left|  \dfrac{\partial}{\partial y_j}\nabla H(1,y)\right| dy 
    \eeaa
    from the mean value theorem,
    there exists a constant $C=C(\gamma)$ such that
    \beaa
    0\le \dfrac{S(x)}{|x|^{\gamma}} \le \sup_{|x|<1}\dfrac{S(x)}{|x|^{\gamma}} + \sup_{|x|\ge 1}S(x)  \le C.
    \eeaa
    Thus, we obtain
    \beaa
    \left| \left(\dfrac{\partial H}{\partial x_j}*\phi\right)(t,x) - \left(\dfrac{\partial H}{\partial x_j}*\phi\right)(t,z) \right| \le C t^{-(1+\gamma)/2} \|\phi\|_{L^{\infty}} |x-z|^{\gamma}.
    \eeaa
    The proof is complete.
\end{proof}

Fix $t\in(0,T]$ arbitrary.
From Lemma \ref{lem:duhamel}, we obtain $u(t)\in BC^1(\bR^n)$.
Moreover, since $f(u,v)$ is differentiable almost everywhere for any open set of $\bR^2$, 
we have 
\beaa
\|\nabla g(t) \|_{L^\infty} < +\infty.
\eeaa
Hence, same argument as in the proof of Lemma \ref{lem:duhamel} can be applied, 
leading to $u(t)\in BC^2(\bR^n)$.
By using Lemma \ref{lem:Holder}, 
we conclude $u(t)\in C^{2+\gamma}(\bR^n)$ for any $\gamma\in (0,1)$.

We utilize the result of the Schauder estimate, Theorem 9.1.2 by  \cite{Krylov}.
Let $\tau \in (0,T)$ and $\gamma\in (0,1)$.
Since $u(\tau)\in C^{2+\gamma}(\bR^n)$ and $g \in C^{\gamma/2,\gamma}([\tau,T]\times \bR^n)$, we obtain that the mild solution satisfies \eqref{eq:rd}
with $u \in C^{1+\gamma/2,2+\gamma}([\tau,T]\times \bR^n)$.
As $\tau$ is arbitrary, it follows that $u \in C^{1,2}((0,T]\times \bR^n)$.
Therefore, the mild solution $u$ becomes the unique classical solution to \eqref{eq:rd}.

\section{Existence of a global solution to 
reaction-diffusion system}
\label{app:rd}

Here, we describe the proof of Proposition \ref{pro:exi-app}.
Let us consider the existence of a mild solution to \eqref{eq:rd} with an initial condition \eqref{init}.

\begin{proposition}
    For any $T>0$, $\delta>0$ and $1\le p\le+\infty$, there exists a unique mild solution $(u^{\delta},v^{\delta}_1,\ldots,v^{\delta}_N)\in \{C([0,T];BC(\bR^n)\cap L^p(\bR^n))\}^{N+1}$ to \eqref{eq:rd} with an initial condition \eqref{init},
    where $u_0\in BC(\bR^n)\cap L^p(\bR^n)$. 
    This solution belongs to $\{C^{1,2}((0,T]\times\bR^n)\}^{N+1}$.
\end{proposition}

This proof is almost same as the argument in Appendix \ref{app:non}.
It should be noted that the following result for the boundedness of the solution is obtained.
\begin{lemma}
For any $q\in [p,+\infty]$, 
\beaa
\| u^{\delta}(t) \|_{L^q}  &\le& \left(1 + \delta C_f \dsum^{N}_{j=1}|\alpha_j|  \right) \|u_0\|_{L^q} \exp\left( C_f \left(1 +  \dsum^{N}_{j=1}|\alpha_j|  \right) t \right),\\
\| v^{\delta}_{j}(t) \|_{L^q}  &\le& e^{-t/\delta}\|u_0\|_{L^q} +\sup_{0\le s\le t} \| u^{\delta}(s)\|_{L^q}, \quad(j=1,2,\ldots,N)
\eeaa
hold for all $t\in [0,T]$.
\end{lemma}
\begin{proof}
    Let $q \in[p,+\infty]$. For any $t\in[0,T]$ and $j=1,2,\ldots,N$, we obtain that
    \beaa
    \|v^{\delta}_{j}(t)\|_{L^q} \le  e^{-t/\delta}\|u_0\|_{L^q} + \dfrac{1}{\delta} \int^{t}_{0} 
 e^{-(t-s)/\delta}\| u^{\delta}(s)\|_{L^q} ds
    \eeaa 
    from the Young inequality. Moreover, we have
    \beaa
    \|u^{\delta}(t)\|_{L^{q}} &\le& \|u_0\|_{L^q} + \int^{t}_{0} \left\| f\left(u^{\delta}(s),\dsum^{N}_{j=1}\alpha_j v^{\delta}_j(s)\right)\right\|_{L^q} ds \\
    &\le& \|u_0\|_{L^q} + C_f \left( \int^{t}_{0} \|u^{\delta}(s)\|_{L^q} ds + \int^{t}_{0}  \dsum^{N}_{j=1}|\alpha_j| \|v^{\delta}_j(s)\|_{L^q} ds \right) \\
    &\le& \|u_0\|_{L^q} + C_f  \int^{t}_{0} \|u^{\delta}(s)\|_{L^q} ds \\ 
    && \quad +C_f\left( \dsum^{N}_{j=1}|\alpha_j|\right) \left( \int^{t}_{0} e^{-s/\delta}\|u_0\|_{L^q} ds +  \dfrac{1}{\delta} \int^{t}_{0} \int^{s}_{0} 
 e^{-(s-\eta)/\delta}\| u^{\delta}(\eta)\|_{L^q} d\eta ds  \right) \\
     &\le& \left(1 + \delta C_f \dsum^{N}_{j=1}|\alpha_j|  \right) \|u_0\|_{L^q} + C_f \left(1 +  \dsum^{N}_{j=1}|\alpha_j|  \right) \int^{t}_{0} \|u^{\delta}(s)\|_{L^q} ds.
    \eeaa
    The Gronwall inequality yields that 
    \beaa
    \|u^{\delta}(t)\|_{L^q} \le  \left(1 + \delta C_f \dsum^{N}_{j=1}|\alpha_j|  \right) \|u_0\|_{L^q} \exp\left( C_f \left(1 +  \dsum^{N}_{j=1}|\alpha_j|  \right) t \right)
    \eeaa
    for all $t\in [0,T]$.
    Finally, we find
    \beaa
    \|v^{\delta}_{j}(t)\|_{L^q} &\le&  e^{-t/\delta}\|u_0\|_{L^q} + \dfrac{1}{\delta} \sup_{0\le s\le t} \| u^{\delta}(s)\|_{L^q} \int^{t}_{0} e^{-(t-s)/\delta} ds \\
    &\le&  e^{-t/\delta}\|u_0\|_{L^q} +  \sup_{0\le s\le t} \| u^{\delta}(s)\|_{L^q}
    \eeaa
    for any $t\in[0,T]$.
\end{proof}




\end{appendices}


\bibliography{sn-bibliography}

\end{document}